\documentclass[12pt,letterpaper]{article}
\usepackage{amsthm,amssymb}
\usepackage{amsfonts}
\makeatletter
\usepackage[dvips]{color}
\usepackage{colordvi,multicol}

\newtheorem{thm}{\textbf Theorem}[section]
\newtheorem{lem}[thm]{\textbf Lemma}
\newtheorem{rem}[thm]{\textbf Remark}

\newtheorem{prop}[thm]{\textbf Proposition}
\newtheorem{defin}[thm]{\textbf Definition}
\newtheorem{assum}[thm]{\textbf Assumption}

\newtheorem{exa}[thm]{\textbf Example}
\newcommand{\be}{\begin{eqnarray*}}
\newcommand{\ee}{\end{eqnarray*}}

\begin{document}

\title{\bf Stochastic Calculus for Markov Processes Associated with Semi-Dirichlet Forms}
  \author{
Chuan-Zhong Chen\\
School of Mathematics and Statistics\\
Hainan Normal University\\
Haikou, 571158, China\\
ccz0082@aliyun.com\\
\\
Li Ma\\
School of Mathematics and Statistics\\
Hainan Normal University\\
Haikou, 571158, China\\
 lma5140377@gmail.com\\
\\
Wei Sun
            \\ Department of Mathematics and Statistics\\
             Concordia University\\
             Montreal, H3G 1M8, Canada\\
             wei.sun@concordia.ca}
   \date{}
\maketitle

\begin{abstract}
\noindent Let $(\mathcal{E},D(\mathcal{E}))$  be a quasi-regular
semi-Dirichlet form and $(X_t)_{t\geq0}$ be the associated Markov
process. For $u\in D(\mathcal{E})_{loc}$, denote $A_t^{[u]}:=
\tilde{u}(X_{t})-\tilde{u}(X_{0})$ and $F^{[u]}_t:=\sum_{0<s\leq
t}(\tilde u(X_{s})-\tilde u(X_{s-}))1_{\{|\tilde u(X_{s})-\tilde
u(X_{s-})|>1\}}$, where $\tilde{u}$ is a quasi-continuous version
of $u$. We show that there exist a unique locally square
integrable martingale additive functional $Y^{[u]}$ and a unique
continuous local additive functional $Z^{[u]}$ of zero quadratic
variation such that
$$
A_t^{[u]}=Y_t^{[u]}+Z_t^{[u]}+F_t^{[u]}.
$$
Further, we define the stochastic integral $\int_0^t\tilde
v(X_{s-})dA_s^{[u]}$ for $v\in D(\mathcal{E})_{loc}$ and derive
the related It\^{o}'s formula.

\end{abstract}
\section[short title]{Introduction}

\noindent Let $({\cal E},D({\cal E}))$ be a quasi-regular
semi-Dirichlet form on $L^2(E;m)$ with associated Markov process
$((X_t)_{t\ge 0}, (P_x)_{x\in E_{\Delta}})$.  For $u\in
{D(\mathcal{E})}_{loc}$, we denote the additive functional (AF in
short) $A^{[u]}$ by
$$
A_t^{[u]}:= \tilde{u}(X_{t})-\tilde{u}(X_{0}), $$ where
$\tilde{u}$ is an ${\cal E}$-quasi-continuous $m$-version of $u$.
In this paper, we will establish a Fukushima type decomposition
for $A^{[u]}$ and study the stochastic integral $\int_0^t\tilde
v(X_{s-})dA_s^{[u]}$ for $v\in D(\mathcal{E})_{loc}$. We refer the
reader to \cite{Fu94}, \cite{MR92}, \cite{MOR} and the next
section for notations and terminologies of this paper.

The celebrated Fukushima's decomposition was originally
established for regular symmetric Dirichlet forms (see \cite{Fu79}
and \cite[Theorem 5.2.2]{Fu94}) and then extended to the
non-symmetric and quasi-regular cases (cf. \cite[Theorem
5.1.3]{oshima} and \cite[Theorem VI.2.5]{MR92}).  If $({\cal
E},D({\cal E}))$ is a quasi-regular Dirichlet form and $u\in
D({\cal E})$, Fukushima's decomposition tells us that there exist
a unique martingale AF (MAF in short) $M^{[u]}$ of finite energy
and a unique continuous AF  (CAF in short) $N^{[u]}$ of zero
energy such that
\begin{equation}\label{Doob}
\tilde{u}(X_t)-\tilde{u}(X_0)=M^{[u]}_t+N^{[u]}_t.
\end{equation}

If $({\cal E},D({\cal E}))$ is a strong local symmetric Dirichlet
form, Fukushima's decomposition (\ref{Doob}) holds  also for $u\in
{D(\mathcal{E})}_{loc}$ with $M^{[u]}$ being a MAF locally of
finite energy and $N^{[u]}$ being a CAF locally of zero energy
(cf. \cite[Theorem 5.5.1]{Fu94}). For a general symmetric
Dirichlet form $({\cal E},D({\cal E}))$, Kuwae showed that the
Fukushima type decomposition holds for a subclass of $
{D(\mathcal{E})}_{loc}$ (see \cite[Theorem 4.2]{kuwae2010}). If
$({\cal E},D({\cal E}))$ is a (not necessarily symmetric)
Dirichlet form, for $u\in {D(\mathcal{E})}_{loc}$, Walsh showed in
\cite{Alexander 2012,Alexander 2013} that there exist a MAF
$W^{[u]}$ locally of finite energy and a CAF $C^{[u]}$ locally of
zero energy such that
\begin{equation}\label{Wa1}
A_t^{[u]}=W_t^{[u]}+C_t^{[u]}+V_t^{[u]},\end{equation} where
$$
V^{[u]}_t:=\sum_{0<s\leq t}(\tilde u(X_{s})-\tilde
u(X_{s-}))1_{\{|\tilde u(X_{s})-\tilde
u(X_{s-})|>1\}}1_{\{t<\zeta\}}-u(X_{\zeta-})1_{\{t\ge\zeta\}}.
$$
Hereafter $\zeta$ denotes the lifetime of $X$.

If $({\cal E},D({\cal E}))$ is only a semi-Dirichlet form, the
situation becomes more complicated. Note that the assumption of
the existence of dual Markov process plays a crucial role in
Fukushima's decomposition. In fact, without that assumption, the
usual definition of energy of AFs is questionable. If $({\cal
E},D({\cal E}))$ is a quasi-regular local semi-Dirichlet form, Ma
et al. showed in \cite{MMS} that Fukushima's decomposition holds
for $u\in {D(\mathcal{E})}_{loc}$. For a general regular
semi-Dirichlet form, Oshima showed in \cite{oshima13} that
Fukushima's decomposition holds for $u\in {D(\mathcal{E})}_b$.

Let $({\cal E},D({\cal E}))$ be a quasi-regular semi-Dirichlet
form. We define $I(\zeta):=[\![0,\zeta[\![\cup[\![\zeta_i]\!]$,
with $\zeta_i$ being the totally inaccessible part of $\zeta$.
Denote by $J$ the jumping measure of $({\cal E},D({\cal E}))$. For
$u\in {D(\mathcal{E})}_{loc}$, Z.M. Ma et al. showed in
\cite[Theorem 1.4]{MSW14} (cf. also \cite{W}) that the following
two assertions are equivalent to each other.

(i) $u$ admits a Fukushima type decomposition. That is, there
exist a locally square integrable MAF $M^{[u]}$ on $I(\zeta)$ and
a local CAF $N^{[u]}$ on $I(\zeta)$ which has zero quadratic
variation such that (\ref{Doob}) holds.

(ii) $u$ satisfies
\begin{eqnarray*}
(S):\  \  \  \mu_u(dx):=\int_E(\tilde u(x)-\tilde u(y))^2J(dy,dx)\
\mbox{is a smooth measure}.
\end{eqnarray*}
Moreover, if $u$ satisfies Condition (S), then the decomposition
(\ref{Doob}) is unique up to the equivalence of local AFs.

In the first part of this paper, we will establish a new Fukushima
type decomposition for $u\in {D(\mathcal{E})}_{loc}$ without
Condition (S). Denote
\begin{equation}\label{V}
F^{[u]}_t:=\sum_{0<s\leq t}(\tilde u(X_{s})-\tilde
u(X_{s-}))1_{\{|\tilde u(X_{s})-\tilde
u(X_{s-})|>1\}}.\end{equation} In Section 2 (see Theorem
\ref{thm3.2} below), we will show that, for any $u\in
{D(\mathcal{E})}_{loc}$, there exist a unique locally square
integrable MAF $Y^{[u]}$ on $I(\zeta)$ and a unique continuous
local AF $Z^{[u]}$ which has zero quadratic variation such that
\begin{eqnarray}\label{semi1}
A_t^{[u]}=Y_t^{[u]}+Z_t^{[u]}+F_t^{[u]}.
\end{eqnarray}
The decomposition (\ref{semi1}) gives the most general form of the
Fukushima type decomposition in the framework of semi-Dirichlet
forms. It implies in particular that $A^{[u]}$ is a Dirichlet
process (cf. \cite{F1,F2}), i.e., summation of a semi-martingale
and a zero quadratic variation process.

In the second part of this paper, we will define the stochastic
integral $\int_0^t\tilde v(X_{s-})dA_s^{[u]}$ for $u,v\in
D(\mathcal{E})_{loc}$ and derive the related It\^{o}'s formula.

Let $({\cal E},D({\cal E}))$ be a regular symmetric Dirichlet
form. For $u\in D({\cal E})$ and $v\in D({\cal E})_b$, Nakao
studied in \cite{N} the stochastic integral $\int_0^t\tilde
v(X_{s-})dA_s^{[u]}$ by introducing the now called Nakao's
integral $\int_0^t\tilde v(X_{s-})dN_s^{[u]}$. Later, Z.Q. Chen
et. al and Kuwae (see \cite {CFKZ} and \cite{kuwae2010}) extended
Nakao's integral to a larger class of integrators as well as
integrands. By using different methods, Walsh (\cite{Alexander
2011}) and C.Z. Chen et al. (\cite{CMS12}) independently extended
Nakao's integral from the setting of symmetric Dirichlet forms to
that of non-symmetric Dirichlet forms. By virtue of the
decomposition (\ref{Wa1}), Walsh also defined Nakao's integral for
more general integrators as well as integrands in the setting of
non-symmetric Dirichlet forms (see \cite{Alexander 2013}). In all
of these references, the related It\^{o}'s formulas have been
derived for the stochastic integral $\int_0^t\tilde
v(X_{s-})dA_s^{[u]}$.

In Section 3, we will  define the stochastic integral
$\int_0^t\tilde v(X_{s-})dA_s^{[u]}$ for $u,v\in
D(\mathcal{E})_{loc}$ and derive the related It\^{o}'s formula in
the setting of semi-Dirichlet forms. Due to the non-Markovian
property of the dual form, all the previous known methods in
defining Nakao's integral ceased to work. Note that if $({\cal
E},D({\cal E}))$ is only a semi-Dirichlet form, its symmetric part
is a symmetric positivity preserving form but in general not a
symmetric Dirichlet form and the dual killing measure might not
exist. These cause extra difficulties in defining Nakao's
integral. In this paper, we will combine the method of
\cite{CMS12} with the localization technique of \cite{MMS} and
\cite{MSW14} to define the stochastic integral $\int_0^t\tilde
v(X_{s-})dA_s^{[u]}$ and derive the related It\^{o}'s formula.

In Section 4, we will give concrete examples of semi-Dirichlet
forms for which our results can be applied.

\section [short title]{Decomposition of ${\tilde u(X_t)-\tilde u(X_0)}$ without Condition (S)}\label{Sec: Decomposition}
\setcounter{equation}{0}

The basic setting of this paper is the same as that in
\cite{MSW14}. We refer the reader to \cite{MSW14} for more
details. Let $E$ be a metrizable Lusin space and $m$ be a
$\sigma$-finite positive measure on its Borel $\sigma$-algebra
${\cal B}(E)$. We consider a quasi-regular semi-Dirichlet form
$({\cal E},D({\cal E}))$ on $L^2(E;m)$. Denote by
$(T_{t})_{t\geq0}$ and $(G_{\alpha})_{\alpha\geq0}$ (resp.
$(\hat{T}_{t})_{t\geq0}$ and $(\hat{G}_{\alpha})_{\alpha\geq0}$)
the semigroup and resolvent (resp. co-semigroup and co-resolvent)
associated with $({\cal E},D({\cal E}))$. Let ${\bf
M}=(\Omega,{\cal F},({\cal F}_t)_{t\ge 0}, (X_t)_{t\ge
0},(P_x)_{x\in E_{\Delta}})$ be an $m$-tight special standard
process which is properly associated with $({\cal E},D({\cal
E}))$.

Throughout this paper, we fix a function $\phi\in L^{1}(E;m)$ with
 $0<\phi\leq1$ $m{\textrm{-}}$a.e. and set $h=G_{1}\phi$,
 $\hat{h}=\hat{G}_{1}\phi$.  Denote $\tau_B:=\inf\{t>0\,|\, X_t\notin B\}$ for $B\subset E$.
Let $V$ be a quasi-open subset of $E$. We denote by
$X^V=(X^V_t)_{t\ge 0}$ the part process of $X$ on $V$ and denote
by $(\mathcal{E}^{{V}},D(\mathcal{E}^{{V}}))$ the part form of
$(\mathcal{E},D(\mathcal{E}))$ on $L^2(V;m)$. It is known that
$X^V$ is a standard process,
$D(\mathcal{E}^{{V}})=D(\mathcal{E})_{{V}}=\{u\in
D(\mathcal{E})\,|\,\tilde u=0,\ {\cal E}{\textrm{-}{\rm q.e.}}\
{\rm on}\ V^c\}$, and $(\mathcal{E}^{{V}},D(\mathcal{E})_{{V}})$
is a quasi-regular semi-Dirichlet form (cf. \cite{kuwae}). Denote
by $({T}_{t}^{V})_{t\ge 0}$, $(\hat{T}_{t}^{V})_{t\ge 0}$,
$({G}_{\alpha}^{V})_{\alpha\ge 0}$ and
$(\hat{G}_{\alpha}^{V})_{\alpha\ge 0}$ the semigroup,
co-semigroup, resolvent and co-resolvent associated with
$(\mathcal{E}^{{V}},D(\mathcal{E})_{{V}})$, respectively. Define
$\bar{h}^V:=\hat{G}^{{V}}_{1}\phi$ and
$\bar{h}^{V,*}:=e^{-2}\hat{T}_{1}^{V}(\hat{G}^{V}_{2}\phi)$. Then
$\bar{h}^V,\bar{h}^{V,*}\in D(\mathcal{E})_{{V}}$ and
$\bar{h}^{V,*}\le \bar{h}^{V}$. Denote
$D(\mathcal{E})_{{V},b}:={\cal B}_b(E)\cap D(\mathcal{E})_{{V}}$.

 For an AF $A=(A_t)_{t\ge 0}$ of $X^V$, we define
$$
e^V(A):=\lim_{t\downarrow0}{1\over{2t}}E_{\bar{h}^V\cdot
m}(A_{t}^2)
$$
whenever the limit exists in $[0,\infty]$. For a local AF
$B=(B_t)_{t\ge 0}$  of $X$, we define
$$e^{V,*}(B):=\lim_{t\downarrow0}{1\over{2t}}E_{\bar{h}^{V,*}\cdot
m}(B_{t\wedge\tau_V}^2)
$$
whenever the limit exists in $[0,\infty]$.

Define
\begin{eqnarray*}
\dot{\mathcal{M}}^V
 &:=&\{M\,|\, M\ \mbox{is an AF of}\ X^{V},\ E_x(M^{2}_t)<\infty, E_x(M_t)=0\\
 &&\ \ \ \ \mbox{for}\ {\rm all}\ t\ge0\ {\rm and}\ {\cal
E}{\textrm{-}{\rm q.e.}}\ x\in V, e^V(M)<\infty\},
 \end{eqnarray*}
  \begin{eqnarray*}
 \mathcal{N}^V_{c}
 &:=&\{N\,|\, N\ \mbox{is a CAF of}\ X^{V},E_x(|N_t|)<\infty\ \mbox{for}\ {\rm all}\ t\ge0\\
   &&\ \ \ \ {\rm and}\ {\cal
E}{\textrm{-}{\rm q.e.}}\ x\in V, e^V(N)=0\},
   \end{eqnarray*}
   \begin{eqnarray*}
\Theta&:=&\{\{{V_n}\}\,|\, {V_n}\ \mbox{is}\ {\cal
E}{\textrm{-}}\mbox{quasi} {\textrm{-}}\mbox{open},\ {V_n}\subset
V_{n+1}\ {\cal E}
{\textrm{-}{\rm q.e.}} \\
         &&\ \ \ \ \ \ \ \ \ \ \ \forall~n\in \mathbb{N},\
         \mbox{and}\ E=\cup_{n=1}^{\infty}{V_n}\ {\cal E}{\textrm{-}{\rm q.e.}}\},
\end{eqnarray*}
\begin{eqnarray*}
{D(\mathcal{E})}_{loc}
&:=&\{u\,|\,\exists\ \{V_n\}\in\Theta\ \mbox{and}\ \{u_n\}\subset D(\mathcal{E})\nonumber\\
  &&\ \ \ \ \ \ \ \mbox{such that }\ u=u_n\ m{\textrm{-}{\rm a.e.}}\ \mbox{on}\ V_n,
  ~\forall~n\in \mathbb{N}\},
\end{eqnarray*}
\begin{eqnarray*}
\dot{\mathcal{M}}_{loc}
 &:=&\{M\,|\, M\ \mbox{is a local AF of}\ {\bf M},\ \exists\ \{V_n\},\{E_n\}\in\Theta\ {\rm and}\ \{M^n\,|\,M^n\in\dot{\mathcal{M}}^{V_n}\}\\
 &&\ \ \ \ \ \ \ \ \mbox{such that}\ E_n\subset V_n,\ M_{t\wedge\tau_{E_n}}=M^{n}_{t\wedge\tau_{E_n}},\ t\ge0,\ n\in\mathbb{N} \}
 \end{eqnarray*}
 and
 \begin{eqnarray*}
{\mathcal{L}}_{c}&:=&\{N\,|\, N\ \mbox{is a local AF of}\ {\bf M}\ ,\ \exists\ \{E_n\}\in\Theta\ \mbox{such that}\ t\rightarrow N_{t\wedge\tau_{E_n}}\\
&&\ \ \ \ \ \ \mbox{  is continuous and of zero quadratic
variation},\  n\in\mathbb{N} \}.
 \end{eqnarray*}
In the above definition, $\{N_{t\wedge\tau_{E_n}}\}$ is said to be
of zero
  quadratic variation if its quadratic variation vanishes in $P_{m}$-measure, more precisely, if it satisfies
$$
\sum_{k=0}^{[T/\varepsilon_l]}(N_{\{(k+1)\varepsilon_l\}\wedge\tau_{E_n}}-N_{\{k\varepsilon_l\}\wedge\tau_{E_n}})^2\rightarrow
0\ {\rm as}\ l\rightarrow\infty\ {\rm in}\ P_{m}{\textrm{-}}{\rm
measure},
$$
 for any $T>0$ and any sequence $\{\varepsilon_l\}_{l\in\mathbb{N}}$ converging to
 0.

We use $\zeta_i$ to denote the totally inaccessible part of
$\zeta$, by which we  mean that $\zeta_i$ is an $\{{\cal
F}_t\}$-stopping time and is the totally inaccessible part of
$\zeta$ w.r.t. $P_{x}$ for ${\cal E}{\textrm{-}{\rm q.e.}}\ x\in
E$. By \cite[Proposition 2.4]{MSW14}, such $\zeta_i$ exists and is
unique in the sense of $P_{x}{\textrm{-}{\rm a.s.}}$ for ${\cal
E}{\textrm{-}{\rm q.e.}}\ x\in E$. We write
$I(\zeta):=[\![0,\zeta[\![\cup[\![\zeta_i]\!]$. By
\cite[Proposition 2.4]{MSW14}, there exists a $\{V_n\}\in\Theta$
such that for any $\{U_n\}\in\Theta$, $I(\zeta)=\cup_n
[\![0,\tau_{V_n\cap U_n}]\!]$ $ P_{x}{\textrm{-}{\rm a.s.}}\ \
{\rm for}\ {\cal E}{\textrm{-}{\rm q.e.}}\ x\in E$. Therefore
$I(\zeta)$ is a predictable set of interval type
 (cf. \cite[Theorem 8.18]{HWY}). By the local compactification method (see \cite[Theorem VI.1.6]{MR92} and \cite[Theorem
3.5]{HC06}) in the semi-Dirichlet forms setting, we may assume
without loss of generality that $(X_t)_{t\ge 0}$ is a Hunt process
and $E$ is a locally compact separable metric space whenever
necessary.

 In this paper a local AF $M$ is called a locally square integrable MAF on $I(\zeta)$,
 denoted by $M\in {\mathcal{M}}^{I(\zeta)}_{loc},$  if $M\in ({\mathcal{M}}^2_{loc})^{I(\zeta)}$ in the sense of \cite[Definition 8.19]{HWY}.  For $u\in
{D(\mathcal{E})}_{loc}$, we define the bounded variation process
$F^{[u]}$ as in (\ref{V}). Denote by $J(dx,dy)$ and $K(dx)$ the
jumping and killing measures of $({\cal E},D({\cal E}))$,
respectively (cf. \cite{HC06}). Let $(N(x,dy),H_s)$ be a L\'{e}vy
system of $X$ and $\mu_H$ be the Revuz measure of the positive ACF
(PCAF in short) $H$. Then we have
$J(dy,dx)=\frac{1}{2}N(x,dy)\mu_H(dx)$ and
$K(dx)=N(x,\Delta)\mu_H(dx)$. Define (cf. \cite[Theorem 5.3]{MMS})
$$
\hat{S}^{*}_{00}:=\{\mu\in S_{0}\,|\,\hat{U}_{1}\mu\leq
c\hat{G}_{1}\phi\ \mbox{for some constant}\ \ c>0\},
$$
where $S_0$ denotes the family of positive measures of finite
energy integral and $\hat{U}_{1}\mu$ is the 1-co-potential.

We put the following assumption:
\begin{assum}\label{assum100}
There exist $\{{V_n}\}\in \Theta$ and locally bounded function
$\{C_n\}$ on $\mathbb{R}$ such that for each $n\in\mathbb{N}$, if
$u,v\in D(\mathcal{E})_{{V_n},b}$ then $uv\in D(\mathcal{E})$ and
\begin{eqnarray*}
\mathcal{E}(uv,uv)\leq C_n(\|u\|_{\infty}+\|v\|_{\infty})({\cal
E}_1(u,u)+{\cal E}_1(v,v)).
\end{eqnarray*}
\end{assum}

Now we can state the main result of this section.
\begin{thm}\label{thm3.2} Let
$(\mathcal{E},D(\mathcal{E}))$ be a quasi-regular semi-Dirichlet
form on $L^{2}(E;m)$ satisfying Assumption \ref{assum100}. Suppose
$u\in {D(\mathcal{E})}_{loc}$. Then,

(i) There exist $Y^{[u]}\in{\mathcal{M}}^{I(\zeta)}_{loc}$ and
$Z^{[u]}\in {\mathcal{L}}_{c}$ such that
\begin{eqnarray}\label{new3}
\tilde{u}(X_{t})-\tilde{u}(X_{0})=Y^{[u]}_{t}+Z^{[u]}_{t}+F^{[u]}_t,\
\ t\ge0,\ \ P_{x}{\textrm{-}a.s.}\ \ {\rm for}\ {\cal
E}{\textrm{-}q.e.}\ x\in E.
\end{eqnarray}
The decomposition (\ref{new3}) is unique up to the equivalence of
local AFs and the continuous part of $M^{[u]}$ belongs to
$\dot{\mathcal{M}}_{loc}$.

(ii)  There exists an $\{E_n\}\in \Theta$ such that for
$n\in\mathbb{N}$, $\{Y^{[u]}_{t\wedge\tau_{E_n}}\}$ is a
$P_{x}$-square-integrable martingale for ${\cal E}$-q.e. $x\in E$,
$e^{E_n,*}(Y^{[u]})<\infty$;
$E_x[(Z^{[u]}_{t\wedge\tau_{E_n}})^2]<\infty$ for $t\ge0$, ${\cal
E}{\textrm{-}}$q.e. $ x\in E$, $e^{E_n,*}(Z^{[u]})=0$.
\end{thm}

A Fukushima type decomposition for $A^{[u]}$ has been established
in \cite{MSW14} under Condition (S). Below we will follow the
argument of \cite{MSW14} to establish the decomposition for
$A^{[u]}-F^{[u]}$ without assuming Condition (S). Before proving
Theorem \ref{thm3.2}, we prepare some lemmas.

We fix a $\{{V_n}\}\in\Theta$ satisfying Assumption
\ref{assum100}. Without loss of generality, we assume that
$\widetilde{\hat{h}}$ is bounded on each ${V_n}$, otherwise we may
replace ${V_n}$ by ${V_n}\cap\{\widetilde{\hat{h}}< n \}$. Since
$\bar{h}^{V_n}=\hat{G}_{1}^{{V_n}}\phi\leq
\hat{G}_{1}\phi=\hat{h}$, $\bar{h}^{V_n}$ is bounded on ${V_n}$.
To simplify notations, we write $${\bar h}_n := {\bar h}^{V_n}.$$

\begin{lem}\label{thm3.6}(\cite[Lemma
1.12]{MSW14}) Let $u\in D(\mathcal{E})_{V_n,b}$. Then there exist
unique $M^{n,[u]}\in\dot{\mathcal{M}}^{V_n}$ and $N^{n,[u]}\in
\mathcal{N}^{V_n}_{c}$ such that ${\rm for}\ {\cal
E}{\textrm{-}q.e.}\ x\in V_n$,
\begin{equation}\label{sdf}
\tilde{u}(X^{{V_n}}_{t})-\tilde{u}(X^{{V_n}}_{0})=M^{n,[u]}_{t}+N^{n,[u]}_{t},\
\ t\ge 0,\ \ P_{x}{\textrm{-}a.s.}
\end{equation}
\end{lem}

We now fix a $u\in {D(\mathcal{E})}_{loc}$. Then, there exist
$\{V^1_n\}\in \Theta$ and $\{u_n\}\subset D(\mathcal{E})$ such
that $u=u_n$ $m{\textrm{-}{\rm a.e.}}$ on $V^1_n$. By
\cite[Proposition 3.6]{MR95}, we may assume without loss of
generality that  each $u_n$ is ${\cal E}$-quasi-continuous. By
\cite[Proposition 2.16]{MR95}, there exists an $\mathcal{E}$-nest
$\{F_{n}^2\}$ of compact subsets of $E$ such that $\{u_n\}\subset
C\{F_{n}^2\}$. Denote by $V^{2}_{n}$ the finely interior of
$F^2_n$. Then $\{V^{2}_{n}\}\in\Theta$. Denote ${V^3_n}=V_n\cap
V^1_n\cap V^2_{n}$. Then $\{{V^3_n}\}\in\Theta$ and each $u_n$ is
bounded on ${V^3_n}$.

For $n\in\mathbb{N}$, we define $E_{n}=\{x\in
E\,|\,{\widetilde{h_n}}(x)>{1\over n}\}$, where
$h_n:=G_{1}^{{V_n}}\phi$. Then $\{E_{n}\}\in\Theta$ satisfying
$\overline{E}_n^{\mathcal{E}}\subset E_{n+1}\ {\cal
E}{\textrm{-}{\rm q.e.}}$ and $E_{n}\subset {V_n}\ {\cal
E}{\textrm{-}{\rm q.e.}}$ for each $n\in \mathbb{N}$ (cf.
\cite[Lemma 3.8]{kuwae}). Here $\overline{E}_n^{\mathcal{E}}$
denotes the ${\cal E}$-quasi-closure of $E_n$. Define
$f_{n}=n\widetilde{h_n}\wedge1$. Then $f_{n}\in
D(\mathcal{E})_{V_n,b}$, $f_{n}=1$ on $E_{n}$ and $f_{n}=0$ on
$V^c_{n}$. Denote by $Q_n$ the bound of $|u_n|$ on $V^3_n$. By
\cite[(2.1)]{kuwae} and Assumption \ref{assum100}, we find that
$[(-Q_nf_n)\vee u_n\wedge (Q_nf_n)]f_n\in D(\mathcal{E})_{V_n,b}$.
To simplify notations, below we use still $u_n$ to denote
$[(-Q_nf_n)\vee u_n\wedge (Q_nf_n)]$. Then we have $u_n,u_nf_n\in
D(\mathcal{E})_{V_n,b}$, and $u=u_n=u_nf_n$ on $E_n\cap V^3_n$.

Denote by $J^n(dx,dy)$ and $K^n$ the jumping and killing measures
of $({\cal E}^{V_n},D({\cal E}^{V_n}))$, respectively. Let
$(N^n(x,dy),H^n_s)$ be a L\'{e}vy system of $X^{V_n}$ and
$\mu_{H^n}$ be the Revuz measure of $H^n$. Then
$J^n(dy,dx)=\frac{1}{2}N^n(x,dy)\mu_{H^n}(dx)$ and
$K^n(dx)=N^n(x,\Delta)\mu_{H^n}(dx)$. For each $n\in \mathbb{N}$,
since $f_n,u_nf_n\in D(\mathcal{E})_{V_n,b}$, $f_n,u_nf_n$ satisfy
Condition (S) by \cite[Proposition 1.8]{MSW14}.  Hence we may
select a $\{V_{n}^4\}\in\Theta$ such that for each
$n\in\mathbb{N}$, $V_{n}^4\subset V_n$, and
$$
\int_{V^4_n}\int_{V_n}(f_n(x)-f_n(y))^{2}J^n(dy,dx)<\infty,
$$
\begin{equation}\label{C}
\int_{V^4_n}\int_{V_n}((u_nf_n)(x)-(u_nf_n)(y))^{2}J^n(dy,dx)<\infty,
\end{equation}
and $$K^n(V^4_n)<\infty.
$$
To simplify notations, we use still $E_n$ to denote $E_n\cap
V^3_n\cap V^4_{n}$. Then we have $\{E_{n}\}\in\Theta$, $E_n\subset
V_n$, $u_nf_n\in D(\mathcal{E})_{V_n,b}$ and $u=u_nf_n$ on $E_n$
for each $n\in \mathbb{N}$.

\begin{lem}\label{lem2.3}
Let $u\in D(\mathcal{E})_{loc}$. Denote
$$
F^{[u],*}:=\sum_{0<s\leq t}(\tilde u(X_s)-\tilde
u(X_{s-}))^21_{\{|\tilde u(X_s)-\tilde u(X_{s-})|\leq1\}}.
$$
Then, $F^{[u],*}_{t\wedge \tau_{E_n}}$ is integrable w.r.t.
$P_{\nu}:=\int P_x\nu(dx)$ for any $\nu\in \hat{S}^{*}_{00}$
satisfying $\nu(E)<\infty$.
\end{lem}
\begin{proof}
Let $\nu\in \hat{S}^{*}_{00}$ with $\nu(E)<\infty$. By \cite[Lemma
A.9]{MMS}, there exists a constant $C_{\nu}>0$ such that for any
PCAF $A$ with Revuz measure $\mu_A$, we have
$$
E_{\nu}(A_t)\le C_{\nu}(1 + t)\int_E\tilde{\hat h}d\mu_A,\ \ t>0.
$$
Therefore,
\begin{eqnarray*}
& &E_{\nu}[F^{[u],*}_{t\wedge \tau_{E_n}}]\\
&\leq&E_{\nu}\left[\sum_{0<s\leq
t\wedge\tau_{E_n}}(u_n(X_s)-u_n(X_{s-}))^21_{\{|u_n(X_s)-u_n(X_{s-})|\leq1\}}\right]+\nu(E)\\
&=&E_{\nu}\left[\int^{t\wedge\tau_{E_n}}_0\int_{E_{\Delta}}[u_n(y)-u_n(X_{s})]^21_{\{|u_n(y)-u_n(X_s)|\leq1\}}
N(X_s,dy)dH_s\right]+\nu(E)\\
&\leq&C_{\nu}(1+t)\int_{E_n}\tilde{\hat{h}}(x)\int_{E_{\Delta}}(u_n(y)-u_n(x))^21_{\{|u_n(y)-u_n(x)|
\leq1\}}N(x,dy)\mu_H(dx)+\nu(E)\\
&=&C_{\nu}(1+t)\left\{2\int_{E_n}\tilde{\hat{h}}(x)\int_{E}(u_n(y)-u_n(x))^21_{\{|u_n(y)-u_n(x)|
\leq1\}}J(dy,dx)\right.\\
& &+\left.\int_{E_n}\tilde{\hat{h}}(x)u_n^2(x)1_{\{|u_n(x)|
\leq1\}}K(dx)\right\}+\nu(E)\\
&=&C_{\nu}(1+t)\left\{2\int_{E_n}\tilde{\hat{h}}(x)\int_{V_n}(u_n(y)-u_n(x))^21_{\{|u_n(y)-u_n(x)|
\leq1\}}J^n(dy,dx)\right.\\
& &+\left.\int_{E_n}\tilde{\hat{h}}(x)u_n^2(x)1_{\{|u_n(x)|
\leq1\}}K^n(dx)\right\}+\nu(E)\\
&\le&C_{\nu}(1+t)\|\tilde{\hat{h}}|_{E_n}\|_{\infty}\left\{2\int_{E_n}f^2_n(x)\int_{V_n}(u_n(y)-u_n(x))^21_{\{|u_n(y)-u_n(x)|
\leq1\}}J^n(dy,dx)\right.\\
& &+\left.\|u_n|_{E_n}\|^2_{\infty}K^n(E_n)\right\}+\nu(E)\\
&\le&C_{\nu}(1+t)\|\tilde{\hat{h}}|_{E_n}\|_{\infty}\left\{4\int_{E_n}\int_{V_n}(f_n(x)-f_n(y))^2J^n(dy,dx)\right.\\
& &+\left.4\int_{E_n}\int_{V_n}f^2_n(y)(u_n(y)-u_n(x))^2J^n(dy,dx)+\|u_n|_{E_n}\|^2_{\infty}K^n(E_n)\right\}+\nu(E)\\
&\le&C_{\nu}(1+t)\|\tilde{\hat{h}}|_{E_n}\|_{\infty}\left\{4\int_{E_n}\int_{V_n}(f_n(x)-f_n(y))^2J^n(dy,dx)\right.\\
& &+\left.8\int_{E_n}\int_{V_n}((u_nf_n)(x)-(u_nf_n)(y))^2J^n(dy,dx)\right.\\
& &+\left.8\int_{E_n}u^2_n(x)\int_{V_n}(f_n(x)-f_n(y))^2J^n(dy,dx)\right.\\
& &+\left.\|u_n|_{E_n}\|^2_{\infty}K^n(E_n)\right\}+\nu(E)\\
&<&\infty.
\end{eqnarray*}
\end{proof}

\noindent \textbf{Proof of Theorem \ref{thm3.2} Assertion (i).}
Let $\{V_n\}$, $\{E_n\}$ and $\{u_nf_n\}$ be given as before. By
Lemma \ref{thm3.6}, for $n\in \mathbb{N}$, there exist unique
$M^{n,[u_nf_n]}\in\dot{\mathcal{M}}^{V_n}$ and $N^{n,[u_nf_n]}\in
\mathcal{N}^{V_n}_{c}$ such that ${\rm for}\ {\cal
E}{\textrm{-}{\rm q.e.}}\ x\in V_n$,
$$
u_nf_n(X^{{V_n}}_{t})-u_nf_n(X^{{V_n}}_{0})=M^{n,[u_nf_n]}_{t}+N^{n,[u_nf_n]}_{t},\
\ t\ge 0,\ \ P_{x}{\textrm{-}{\rm a.s.}}
$$
Hereafter, for a martingale $M$, we denote by $M^c$ and $M^d$ its
continuous part and purely discontinuous part, respectively. By
\cite[Lemma 1.14]{MSW14}, for $n<l$, we have
 $M^{n,[u_nf_n],c}_{t\wedge\tau_{E_n}}=M^{l,[u_lf_l],c}_{t\wedge\tau_{E_n}}$
, $t\ge0$, $P_{x}{\textrm{-}{\rm a.s.}}\ \ {\rm for}\ {\cal
E}{\textrm{-}{\rm q.e.}}\ x\in  V_n$. Therefore, we can define
$M^{[u],c}_{t\wedge\tau_{E_n}}:=\lim_{l\rightarrow\infty}M^{l,[u_lf_l],c}_{t\wedge\tau_{E_n}}$
and $M^{[u],c}_t:=0$ for $t>\zeta$  if there exists some $n$ such
that $\tau_{E_n}=\zeta$ and $\zeta<\infty$; or $M^{[u],c}_t:=0$
for $t\ge\zeta$, otherwise. Following the argument of the proof of
\cite [Theorem 1.4]{MSW14}, we can show that $M^{[u],c}$ is well
defined, $M^{[u],c}\in \dot{\mathcal{M}}_{loc}$ and
$M^{[u],c}\in{\mathcal{M}}^{I(\zeta)}_{loc}$.

Denote $\Delta u(X_s):=\tilde u(X_s)-\tilde u(X_{s-})$. By Lemma
\ref{lem2.3},
\begin{eqnarray*}
Y^l_t:&=&\sum_{0<s\leq t}\Delta u(X_s)I_{\{\frac{1}{l}\leq|\Delta
u(X_s)|\leq1\}}-
\left(\sum_{0<s\leq t}\Delta u(X_s)I_{\{\frac{1}{l}\leq|\Delta u(X_s)|\leq1\}}\right)^p\\
&=&\sum_{0<s\leq t}\Delta u(X_s)I_{\{\frac{1}{l}\leq|\Delta u(X_s)|\leq1\}}\\
&&-\int^{t}_0\int_{\{\frac{1}{l}\leq|\tilde u(y)-\tilde
u(X_s)|\leq1\}}(\tilde u(y)-\tilde u(X_{s}))N(X_s,dy)dH_s
\end{eqnarray*}
is well-defined. Hereafter $^p$ denotes the dual predictable
projection. Further, by Lemma \ref{lem2.3} and following the
argument of the proof of \cite [Theorem 1.4]{MSW14} (with $M^l$
therein replaced with $Y^l$ of this paper), we can show that for
${\cal E}\textrm{-}{\rm q.e.}\ x\in E$,
$Y^{l_k}_{t\wedge\tau_{E_n}}$ converges uniformly in $t$ on each
finite interval for a subsequence $\{l_k\rightarrow\infty\}$ (and
hence for the whole sequence $\{k\}$) and for each $k$,
\begin{eqnarray*}
Y^{l_k}_{(t+s)\wedge\tau_{E_n}}=Y^{l_k}_{t\wedge\tau_{E_n}}+Y^{l_k}_{s\wedge\tau_{E_n}}\circ\theta_{t\wedge\tau_{E_n}},\
\mbox{if} \ 0\leq t,s<\infty.
\end{eqnarray*}
Thus,  $L^{n}$, the limit of
$\{Y^{l_{k}}_{s\wedge\tau_{E_n}}\}_{k=1}^{\infty}$, is a
$P_x$-square integrable purely  discontinuous martingale for
${\cal E}\textrm{-}{\rm q.e.}\ x\in E$ and satisfies:
\begin{eqnarray*}
L^{n}_{(t+s)\wedge\tau_{E_n}}=L^{n}_{t\wedge\tau_{E_n}}+L^{n}_{s\wedge\tau_{E_n}}\circ\theta_{t\wedge\tau_{E_n}},\
\mbox{if} \ 0\leq t,s<\infty.
\end{eqnarray*}
 By the above construction, we
find that
$L^{n_{1}}_{t\wedge\tau_{E_{n_{1}}}}=L^{n_{2}}_{t\wedge\tau_{E_{n_{1}}}}$
for $n_{1}\leq n_{2}$. We define
$Y^{[u],d}_{t}=L^{n}_{t},t\leq\tau_{E_{n}}$, and
$Y^{[u],d}_{t}=L^{n}_{t},t\geq\zeta$, if for some $n$,
$\tau_{E_n}=\zeta<\infty$; $Y^{[u],d}_{t}=0,t\geq\zeta$,
otherwise. Then $Y^{[u],d}\in{\mathcal{M}}^{I(\zeta)}_{loc}$,
which gives all the jumps of $\tilde{u}(X_t)-\tilde{u}(X_0)$ on
$I(\zeta)$ with jump size less than or equal to 1. Since
$\{Y^{l}_{t}\}$ is an MAF for each $l$, we find that
$\{Y^{[u],d}_{t}\}$ is a local MAF by the uniform convergence on
$I(\zeta)$.

We define $Y^{[u]}:=M^{[u],c}+Y^{[u],d}$ and
$Z^{[u]}_{t\wedge\tau_{E_n}}:=\tilde{u}(X_{t\wedge\tau_{E_n}})-\tilde{u}(X_0)-Y^{[u]}_{t\wedge\tau_{E_n}}-F^{[u]}_{t\wedge\tau_{E_n}}$
for each $n\in\mathbb{N}$. Then $Z^{[u]}$ is a local AF of ${\bf
M}$ and $t\mapsto Z^{[u]}_{t\wedge\tau_{E_n}}$ is continuous. Now
we show that $\{Z^{[u]}_{t\wedge\tau_{E_n}}\}$ has zero quadratic
variation  and hence $Z^{[u]}\in {\mathcal{L}}_{c}$. By
Fukushima's decomposition for part processes, we have that
\begin{eqnarray}\label{Vs}
& &{u_nf_{n}}(X_{t\wedge\tau_{E_n}})-{u_nf_{n}}(X_{0})\nonumber\\
&=&{u_nf_{n}}(X^{V_{n}}_{t\wedge\tau_{E_n}})-{u_nf_{n}}(X^{V_{n}}_{0})\nonumber\\
&=&M^{n,[u_nf_{n}]}_{t\wedge\tau_{E_n}}+N^{n,[u_nf_{n}]}_{t\wedge\tau_{E_n}}\nonumber\\
&=&M^{n,[u_nf_{n}],c}_{t\wedge\tau_{E_n}}+M^{n,[u_nf_{n}],d}_{t\wedge\tau_{E_n}}+N^{n,[u_nf_{n}]}_{t\wedge\tau_{E_n}}\nonumber\\
&=&M^{n,[u_nf_{n}],c}_{t\wedge\tau_{E_n}}+M^{n,[u_nf_{n}],sd}_{t\wedge\tau_{E_n}}+M^{n,[u_nf_{n}],bd}_{t\wedge\tau_{E_n}}
+N^{n,[u_nf_{n}]}_{t\wedge\tau_{E_n}},
\end{eqnarray}
where
\begin{eqnarray*} M^{n,[u_nf_n],sd}_t
&=&\lim_{l\rightarrow\infty}\left\{
\sum_{0<s\leq t}(u_nf_n(X^{V_n}_s)-u_nf_n(X^{V_n}_{s-}))1_{\{\frac{1}{l}\le|u_nf_n(X^{V_n}_s)-u_nf_n(X^{V_n}_{s-})|\leq1\}}\right.\\
&&\left.-\int^{t}_0\int_{\{\frac{1}{l}\le|u_nf_n(y)-u_nf_n(X^{V_n}_{s})|\leq1\}}(u_nf_n(y)-u_nf_n(X^{V_n}_{s}))N^n(X^{V_n}_s,dy)dH^n_s\right\},
\end{eqnarray*}
and
\begin{eqnarray*} M^{n,[u_nf_n],bd}_t
&=&\sum_{0<s\leq t}(u_nf_n(X^{V_n}_s)-u_nf_n(X^{V_n}_{s-}))1_{\{|u_nf_n(X^{V_n}_s)-u_nf_n(X^{V_n}_{s-})|>1\}}\\
&&-\int^{t}_0\int_{\{|u_nf_n(y)-u_nf_n(X^{V_n}_{s})|>1\}}(u_nf_n(y)-u_nf_n(X^{V_n}_{s}))N^n(X^{V_n}_s,dy)dH^n_s.
\end{eqnarray*}

We define
\begin{eqnarray*}
B_t&:=&\left\{(\tilde u(X_{\tau_{E_n}})-\tilde u(X_{\tau_{E_n}-}))1_{\{|\tilde u(X_{\tau_{E_n}})-\tilde u(X_{\tau_{E_n}-})|\leq1\}}\right.\\
&&\left.-(u_nf_n(X_{\tau_{E_n}})-u_nf_n(X_{\tau_{E_n}-}))1_{\{|u_nf_n(X_{\tau_{E_n}})-u_nf_n(X_{\tau_{E_n}-})|\leq1\}}\right\}
1_{\{\tau_{E_n\leq t}\}}.
\end{eqnarray*}
$\{B_{t}\}$ is an adapted quasi-left continuous bounded variation
processes and hence its dual predictable projection
$\{B^{p}_{t}\}$
 is an adapted continuous bounded variation
processes (cf. \cite[Theorem A.3.5]{Fu94}). By comparing
(\ref{Vs}) to
$$\tilde{u}(X_{t\wedge\tau_{E_n}})-\tilde{u}(X_0)=M^{[u],c}_{t\wedge\tau_{E_n}}+Y^{[u],d}_{t\wedge\tau_{E_n}}
+Z^{[u]}_{t\wedge\tau_{E_n}}+F^{[u]}_{t\wedge\tau_{E_n}},
$$
 we get
\begin{eqnarray}\label{newasd}
Z^{[u]}_{t\wedge\tau_{E_n}}
&=&N^{n,[u_nf_{n}]}_{t\wedge\tau_{E_n}}+M^{n,[u_nf_{n}],sd}_{t\wedge\tau_{E_n}}-Y^{[u],d}_{t\wedge\tau_{E_n}}+M^{n,[u_nf_{n}],bd}_{t\wedge\tau_{E_n}}
-F^{[u]}_{t\wedge\tau_{E_n}}\nonumber\\
& &+\tilde{u}(X_{t\wedge\tau_{E_n}})
-{u_nf_{n}}(X_{t\wedge\tau_{E_n}})\nonumber\\
&=&N^{n,[u_nf_{n}]}_{t\wedge\tau_{E_n}}+(M^{n,[u_nf_{n}],sd}_{t\wedge\tau_{E_n}}-Y^{[u],d}_{t\wedge\tau_{E_n}}+B_t-B^p_t)+B^p_t\nonumber\\
&
&-\int^{t\wedge\tau_{E_n}}_0\int_{\{|u_nf_n(y)-u_nf_n(X^{V_n}_{s})|>1\}}(u_nf_n(y)-u_nf_n(X^{V_n}_{s}))\nonumber\\
& &\ \ \ \ \ \ \ \ \ \ \ \ \ \ \ \ \ \cdot
N^n(X^{V_n}_s,dy)dH^n_s.
\end{eqnarray}
Hence
$\{M^{n,[u_nf_{n}],sd}_{t\wedge\tau_{E_n}}-Y^{[u],d}_{t\wedge\tau_{E_n}}+B_t-B^p_t\}$
is a purely discontinuous martingale with zero jump, which must be
equal to zero. The quadratic variations of
$N^{n,[u_nf_{n}]}_{t\wedge\tau_{E_n}}$ and $B^p_t$ vanish in
$P_{\bar{h}_{n}\cdot m}{\textrm{-}}{\rm measure}$ and
$P_{\phi\cdot m}{\textrm{-}}{\rm measure}$, respectively. Denote
by $C^n_t$ the last term of (\ref{newasd}). By (\ref{C}),
 one finds that $\{C^n_t\}$ is a $P_{\nu}$-square-integrable continuous bounded variation process for any $\nu\in \hat{S}^{*}_{00}$
satisfying $\nu(E)<\infty$. Hence its quadratic variation vanishes
in $P_{\phi\cdot m}{\textrm{-}}{\rm measure}$. Therefore, the
quadratic variation of $\{Y^{[u]}_{t\wedge\tau_{E_n}}\}$ vanishes
in $P_m{\textrm{-}}{\rm measure}$ since $m(E_n)<\infty$, i.e.,
$\{Y^{[u]}_{t\wedge\tau_{E_n}}\}$ has zero quadratic variation.

Finally, we prove the uniqueness of decomposition (\ref{new3}).
Suppose that $Y'\in \mathcal{M}^{I(\zeta)}_{loc}$ and
$Z'\in\mathcal{L}_c$ such that
 \begin{eqnarray*}
\tilde u(X_t)-\tilde u(X_0)=Y'_t+Z'_t+F^{[u]}_t,\ \ t\ge0,\ \
P_{x}{\textrm{-}{\rm a.s.}}\ \ {\rm for}\ {\cal E}{\textrm{-}{\rm
q.e.}}\ x\in E.
\end{eqnarray*}
By \cite[Proposition 2.4]{MSW14}, we can choose an
$\{E_n\}\in\Theta$ such that $I(\zeta)=\cup_n
[\![0,\tau_{E_n}]\!]$ $ P_{x}{\textrm{-}{\rm a.s.}}\ \ {\rm for}\
{\cal E}{\textrm{-}{\rm q.e.}}\ x\in E$. Then, for each
$n\in\mathbb{N}$, $\{(Y^{[u]}-Y')^{\tau_{E_n}}\}$ is a locally
square integrable martingale and a zero quadratic variation
process. This implies that
$P_m(\langle(Y^{[u]}-Y')^{\tau_{E_n}}\rangle_t=0, \forall
t\in[0,\infty))=0$. Consequently by the analog of \cite[Lemma
5.1.10]{Fu94} in the semi-Dirichlet forms setting,
$P_x(\langle(Y^{[u]}-Y')^{\tau_{E_n}}\rangle_t=0, \forall
t\in[0,\infty))=0$ ${\rm for}\ {\cal E}{\textrm{-}{\rm q.e.}}\
x\in E$. Therefore $Y_t^{[u]}=Y'_t$, $0\le t\le\tau_{E_n}$,
$P_{x}{\textrm{-}{\rm a.s.}}\ {\rm for}\ {\cal E}{\textrm{-}{\rm
q.e.}}\ x\in E$. Since $n$ is arbitrary, we obtain the uniqueness
of decomposition (\ref{new3}) up to the equivalence of local AFs.

\noindent \textbf{Proof of Theorem \ref{thm3.2} Assertion (ii).}
By (i), $Y^{[u]}\in{\mathcal{M}}^{I(\zeta)}_{loc}$. Hence $\langle
Y^{[u],d}\rangle_t=(\int_0^{t}\int_{E_{\Delta}}(\tilde
u(X_s)-\tilde u(y))^21_{\{|\tilde u(X_s)-\tilde
u(y)|\leq1\}}N(X_s,dy)dH_s)1_{I(\zeta)}$ is a PCAF on $I(\zeta)$
and can be extended to a PCAF by \cite[Remark 2.2]{CFKZ}.  The
Revuz measure of $\langle Y^{[u],d}\rangle$ is given by
\begin{eqnarray*} \mu_{\langle
u\rangle}^{d}(dx)&=&2\int_{E}(\tilde u(x)-\tilde
u(y))^21_{\{|\tilde
u(x)-\tilde u(y)|\leq1\}}J(dy,dx)\\
& &+\tilde {u}^2(x)1_{\{|\tilde u(x)|\leq1\}}K(dx).
\end{eqnarray*}
By \cite[Lemma 1.1]{MSW14}, $\mu_{\langle u\rangle}^{d}$ is a
smooth measure. Therefore, there exists an $\{E'_n\}\in \Theta$
such that $ \mu_{\langle u\rangle}^{d}(E'_n)<\infty$ for each
$n\in \mathbb{N}$. To simplify notations, we use still $E_n$ to
denote $E_n\cap E'_n$. The remaining part of the proof is similar
to that of \cite[Theorem 1.15]{MSW14}. We omit the details here.
\hfill\fbox

 \begin{rem}
(i) As in \cite[Theorem 1.4]{MSW14}, if we use
${\mathcal{M}}^{[\![0,\zeta[\![}_{loc}$ instead of
${\mathcal{M}}^{I(\zeta)}_{loc}$ , then the uniqueness of the
decomposition (\ref{new3}) may fail to be true.

(ii) For $u\in D(\mathcal{E})_{loc}$, if Condition (S) holds,
i.e., $\mu_{u}\in S$, then by \cite[Theorem 1.4]{MSW14}, there
exist unique $M^{[u]}\in{\mathcal{M}}^{I(\zeta)}_{loc}$ and
$N^{[u]}\in {\mathcal{L}}_{c}$ such that
\begin{equation}\label{Vs2}
\tilde{u}(X_{t})-\tilde{u}(X_{0})=M^{[u]}_{t}+N^{[u]}_{t},\ \
t\ge0,\ \ P_{x}{\textrm{-}a.s.}\ \ {for}\ {\cal
E}{\textrm{-}q.e.}\ x\in E.
\end{equation}
with\begin{equation}\label{Vs3}
M^{[u]}_{t}=M^{[u],c}_{t}+M^{[u],d}_{t},
\end{equation}
and
\begin{eqnarray}\label{Vs4} M^{[u],d}_{t}
&=&\lim_{l\rightarrow\infty}\left\{
\sum_{0<s\leq t}(\tilde u(X_s)-\tilde u(X_{s-}))1_{\{\frac{1}{l}\le|\tilde u(X_s)-\tilde u(X_{s-})|\}}\right.\nonumber\\
&&\left.-\int^{t}_0\int_{\{\frac{1}{l}\le|\tilde u(y)-\tilde
u(X_{s})|\}}(\tilde u(y)-\tilde u(X_{s}))N(X_s,dy)dH_s\right\}.
\end{eqnarray}
 By comparing
(\ref{Vs2})-(\ref{Vs4}) with
\begin{eqnarray*}\tilde{u}(X_t)-\tilde{u}(X_0)&=&Y^{[u]}_{t}
+Z^{[u]}_{t}+F^{[u]}_{t}\\
&=&M^{[u],c}_{t}+Y^{[u],d}_{t} +Z^{[u]}_{t}+F^{[u]}_{t},
\end{eqnarray*}
\begin{eqnarray*} Y^{[u],d}_{t}
&=&\lim_{l\rightarrow\infty}\left\{
\sum_{0<s\leq t}(\tilde u(X_s)-\tilde u(X_{s-}))1_{\{\frac{1}{l}\le|\tilde u(X_s)-\tilde u(X_{s-})|\le 1\}}\right.\nonumber\\
&&\left.-\int^{t}_0\int_{\{\frac{1}{l}\le|\tilde u(y)-\tilde
u(X_{s})|\le 1\}}(\tilde u(y)-\tilde
u(X_{s}))N(X_s,dy)dH_s\right\},
\end{eqnarray*}
we get
\begin{eqnarray*}
M^{[u]}_t&=&Y^{[u]}+\sum_{0<s\leq t}(\tilde u(X_s)-\tilde
u(X_{s-}))1_{\{|\tilde u(X_s)-\tilde u(X_{s-})|>1\}}\\
&&-\int^{t}_0\int_{\{|\tilde u(y)-\tilde u(X_{s})|>1\}}(\tilde
u(y)-\tilde u(X_{s}))N(X_s,dy)dH_s,
\end{eqnarray*}
and
$$
N^{[u]}_t=Z^{[u]}+\int^{t}_0\int_{\{|\tilde u(y)-\tilde
u(X_{s})|>1\}}(\tilde u(y)-\tilde u(X_{s}))N(X_s,dy)dH_s.
$$

\end{rem}

\section [short title]{Stochastic Integral and It\^o's formula}\label{Sec:Fukushima}
\setcounter{equation}{0}

Let $({\cal E},D({\cal E}))$ be a quasi-regular semi-Dirichlet
form on $L^2(E;m)$ with associated Markov process $((X_t)_{t\ge
0}, (P_x)_{x\in E_{\Delta}})$. Throughout this section, we put the
following assumption.
 \begin{assum}\label{assum1q}
There exist $\{{V_n}\}\in \Theta$,  Dirichlet forms
$(\eta^{(n)},D(\eta^{(n)}))$ on $L^2(V_n;m)$, and constants
$\{C_n>1\}$ such that  for each $n\in\mathbb{N}$, $D(\eta^{(n)})=
D(\mathcal{E})_{{V_n}}$ and
\begin{eqnarray*}
\frac{1}{C_n}\eta^{(n)}_1(u,u)\leq\mathcal{E}_1(u,u)\leq
C_n\eta^{(n)}_1(u,u), \ \ \forall u\in D(\mathcal{E})_{{V_n}}.
\end{eqnarray*}
\end{assum}

Obviously, Assumption \ref{assum1q} implies Assumption
\ref{assum100}. In this section, we will first define stochastic
integrals for part forms $(\mathcal{E}^{V_n},D({\cal{E}})_{V_n})$
and then extend them to $({\cal E},D({\cal E}))$.

\subsection {Stochastic Integral for Part Process}

We fix a $\{{V_n}\}\in \Theta$ satisfying Assumption
\ref{assum1q}. Without loss of generality, we assume that
$\widetilde{\hat{h}}$ is bounded on each ${V_n}$, otherwise we may
replace ${V_n}$ by ${V_n}\cap\{\widetilde{\hat{h}}< n \}$. For
$n\in\mathbb{N}$, let $(\mathcal{E}^{V_n},D({\cal{E}})_{V_n})$ be
the part form of $(\mathcal{E},D({\cal{E}}))$ on $L^2(V_n;m)$.
Then $(\mathcal{E}^{V_n},D({\cal{E}})_{V_n})$ is a quasi regular
semi-Dirichlet form with associated Markov process
$((X^{V_n}_t)_{t\ge 0},(P^{V_n}_x)_{x\in (V_n)_{\Delta}})$.

Let $u\in D(\mathcal{E})_{V_n}$ and denote $A_t^{n,[u]}=
\tilde{u}(X^{V_n}_{t})-\tilde{u}(X^{V_n}_{0})$. By Lemma
\ref{thm3.6}, we have the decomposition (\ref{sdf}). For $v\in
D(\mathcal{E})_{V_n,b}$, we will follow \cite{CMS12} to define the
stochastic integral $\int_0^t\tilde v(X^{V_n}_{s-})dA_s^{n,[u]}$
and derive the related It\^o's formula. Note that since
$(\mathcal{E}^{V_n},D({\cal{E}})_{V_n})$ is only a semi-Dirichlet
form, its symmetric part $(\tilde{\cal
E}^{V_n},D({\cal{E}})_{V_n})$ might not be a Dirichlet form.
However, we can use $(\tilde{\cal \eta}^{(n)},D({\cal
\eta}^{(n)}))$, the symmetric part of $({\cal \eta}^{(n)},D({\cal
\eta}^{(n)}))$, to substitute $(\tilde{\cal
E}^{V_n},D({\cal{E}})_{V_n})$ and then follow the argument of
\cite{CMS12} to define Nakao's integral $\int_0^t\tilde
v(X^{V_n}_{s-})dN_s^{n,[u]}$ and prove its related properties.
Below we will mainly state the results and point out only the
necessary modifications in proofs. For more details we refer the
reader to \cite{CMS12}.

Similar to \cite[Lemma 2.1]{CMS12}, we can prove the following
lemma.
\begin{lem}\label{lemman1}
Let $f\in D({\cal E})_{V_n}$. Then there exist unique $f^*\in D({\cal E})_{V_n}$ and $f^{\triangle}\in D({\cal E})_{V_n}$ such that for any $g\in D({\cal E})_{V_n}$,
\begin{equation}\label{1}
{\cal E}^{V_n}_1(f,g)=\tilde{\cal \eta}^{(n)}_1(f^*,g)
\end{equation}
and
\begin{equation}\label{2}
\tilde{\cal \eta}^{(n)}_1(f,g)={\cal E}^{V_n}_1(f^{\triangle},g).
\end{equation}
\end{lem}

Let $f,g\in D({\cal E})_{V_n}$. We use $\tilde{\mu}^{(n)}_{\langle
f,g\rangle}$ to denote the mutual energy measure of $f$ and $g$
w.r.t. the symmetric Dirichlet form $(\tilde{\cal
\eta}^{(n)},D({\cal E})_{V_n})$. Suppose that $u\in D({\cal
E})_{V_n}$ and $v\in D({\cal E})_{V_n,b}$. It is easy to see that
there exists a unique element in $D({\cal E})_{V_n}$, which is
denoted by $\lambda(u,v)$, such that
$$
\frac{1}{2}\int_{V_n} \tilde vd\tilde{\mu}^{(n)}_{\langle
h,u^*\rangle}=\tilde{\cal \eta}^{(n)}_1(\lambda(u,v),h),\ \
\forall h\in D({\cal E})_{V_n}.
$$
Let  $u^*$ and $\lambda(u,v)^{\triangle}$ be the unique elements
in $D({\cal E})_{V_n}$ as defined by (\ref{1}) and (\ref{2})
relative to $u$ and $\lambda(u,v)$, respectively. Similar to
\cite[Theorem 2.2]{CMS12}, we can prove the following result.
\begin{thm}
Let $u\in D({\cal E})_{V_n}$ and $v\in D({\cal E})_{V_n,b}$. Then, for any $h\in D({\cal E})_{V_n,b}$,
\begin{equation}\label{11}
{\cal E}^{V_n}(u,hv)={\cal
E}^{V_n}_1(\lambda(u,v)^{\triangle},h)+\frac{1}{2}\int_{V_n}
\tilde hd\tilde{\mu}^{(n)}_{\langle v,u^*\rangle}+\int_{V_n}
(u^*-u)hvdm.
\end{equation}
\end{thm}

Denote by $A_c^{n,+}$ the family of all PCAFs  of $X^{V_n}$.
Define $$A_c^{n,+,f}:=\{A\in A_c^{n,+}\,|\,{\rm the\ smooth\
measure},\ \mu_A,\ {\rm cossreponding\ to}\ A\ {\rm is\
finite}\}$$ and
\begin{eqnarray*}
{\cal N}^{n,*}_c:=&&\{N^{[u]}_t+\int_0^tf(X_s)ds+A^{(1)}_t-A^{(2)}_t\,|\,u\in D({\cal E})_{V_n},f\in L^2(V_n;m)\ {\rm and}\\
&&\ \ \ \ \ \ A^{(1)},A^{(2)}\in A_c^{n,+,f}\}.
\end{eqnarray*}
Note that any $C\in {\cal N}^{n,*}_c$ is finite and continuous on
$[0,\infty)$ $P_{x}{\textrm{-}{\rm a.s.}}$ for ${\cal
E}{\textrm{-}{\rm q.e.}}\ x\in E$. Similar to \cite[Theorem
2.2]{N}, we can prove the following lemma.
\begin{lem}\label{lem2} Let $\Upsilon$ be a finely open set such
that $\Upsilon\subset V_n$.  If $C^{(1)}, C^{(2)}\in {\cal
N}^{n,*}_c$ satisfying
$$
\lim_{t\downarrow 0}\frac{1}{t}E^{V_n}_{h\cdot
m}[C^{(1)}_t]=\lim_{t\downarrow 0} \frac{1}{t}E^{V_n}_{h\cdot
m}[C^{(2)}_t],\ \ \forall h\in D({\cal E})_{\Upsilon,b},
$$
then $C^{(1)}=C^{(2)}$ for $t\le \tau_{\Upsilon}$
$P^{V_n}_{x}{\textrm{-}a.s.}$ for ${\cal E}{\textrm{-}q.e.}\ x\in
V_n$.
\end{lem}
Note that $\tilde{\mu}^{(n)}_{\langle v,u^*\rangle}$ is a signed
smooth measure w.r.t. $(\tilde{\mathcal{
\eta}}^{(n)},D(\eta^{(n)}))$ and hence $({\cal E}^{V_n},D({\cal
E})_{V_n})$ by Assumption \ref{assum1q}. We use $G(u,v)$ to denote
the unique element in $A_c^{n,+}-A_c^{n,+}$ that is corresponding
to $\tilde{\mu}^{(n)}_{\langle v,u^*\rangle}$ under the Revuz
correspondence between smooth measures of $({\cal E}^{V_n},D({\cal
E})_{V_n})$ and PCAFs of $X^{V_n}$ (cf. \cite[Theorem A.8]{MMS}).
To simplify notations, we define
$$
\Gamma(u,v)_t:=N^{[\lambda(u,v)^{\triangle}]}_t-\int_0^t\lambda(u,v)^{\triangle}(X^{V_n}_s)ds,\ \ t\ge0.
$$

\begin{defin}\label{00s} Let $u\in D({\cal E})_{V_n}$ and $v\in D({\cal E})_{V_n,b}$. We define for $t\ge 0$,
\begin{eqnarray*}
\int_0^t\tilde v(X^{V_n}_{s-})dN^{n,[u]}_s&:=&\int_0^t\tilde v(X^{V_n}_s)dN^{n,[u]}_s\\
&:=&\Gamma(u,v)_t-\frac{1}{2}G(u,v)_t-\int_0^t(u^*-u)v(X^{V_n}_s)ds.
\end{eqnarray*}
\end{defin}

\begin{rem}\label{rem1}
Let $u\in D({\cal E})_{V_n}$ and $v\in D({\cal E})_{V_n,b}$. Then
one can check that $\int_0^t\tilde v(X^{V_n}_s)dN^{n,[u]}_s\in
{\cal N}^{n,*}_c$. By Definition \ref{00s}, (\ref{sdf}),
\cite[Theorem 3.4]{f}, \cite[Theorem A.8(iii)]{MMS} and
(\ref{11}), we obtain that
$$
\lim_{t\downarrow 0}\frac{1}{t}E^{V_n}_{h\cdot
m}\left[\int_0^t\tilde v(X^{V_n}_s)dN^{[u],n}_s\right]=-{\cal
E}^{V_n}(u,hv).
$$
Therefore, by Lemma \ref{lem2}, $\int_0^t\tilde
v(X^{V_n}_s)dN^{n,[u]}_s$ is the unique AF $(C_t)_{t\ge 0}$ in
${\cal N}^{n,*}_c$ that satisfies $\lim_{t\downarrow
0}\frac{1}{t}E^{V_n}_{h\cdot m}[C_t]=-{\cal E}^{V_n}(u,hv)$.
\end{rem}

Denote by $(L^{V_n},D(L^{V_n}))$ the generator of $({\cal
E}^{V_n},D({\cal E})_{V_n})$. Note that if $u\in D(L^{V_n})$ then
$dN^{n,[u]}_s=L^{V_n}u(X^{V_n}_s)ds$.  In this case, it is easy to
see that for any $v,h\in D({\cal E})_{V_n,b}$,
$$
\lim_{t\downarrow 0}\frac{1}{t}E^{V_n}_{h\cdot m}\left[\int_0^tv(X^{V_n}_s)L^{V_n}u(X^{V_n}_s)ds\right]=\int_{V_n}hvL^{V_n}udm=-{\cal E}^{V_n}(u,hv)
$$
(cf. \cite[Theorem A.8(vi)]{MMS}).  Hence our definition of the
stochastic integral $\int_0^t\tilde v(X^{V_n}_s)dN^{n,[u]}_s$ for
$u\in D({\cal E})_{V_n}$ is an extension of the ordinary Lebesgue
integral $\int_0^t\tilde v(X_s^{V_n})L^{V_n}u(X^{V_n}_s)ds$ for
$u\in D(L^{V_n})$. More generally, similar to \cite[Proposition
2.6]{CMS12}, we can prove the following proposition.
\begin{prop}\label{pro2.7} Let $u\in D({\cal E})_{V_n}$, $v\in D({\cal E})_{V_n,b}$ and $\Upsilon$ be a finely open set such
that $\Upsilon\subset V_n$. Suppose that there exist
$A^{(1)},A^{(2)}\in A_c^{n,+}$ such that
$N^{n,[u]}_t=A^{(1)}_t-A^{(2)}_t$ for $t<\tau_{\Upsilon}$ (resp.
all $t\ge 0$). Then
$$
\int_0^t\tilde v(X^{V_n}_s)dN^{n,[u]}_s=\int_0^t\tilde
v(X^{V_n}_s)d(A^{(1)}_s-A^{(2)}_s)\ {\rm for}\
t\le\tau_{\Upsilon}\ ({\rm resp.\ all}\ t\ge 0)
$$
$P^{V_n}_{x}{\textrm{-}a.s.}$ for ${\cal E}{\textrm{-}q.e.}\ x\in
V_n$.
\end{prop}

\begin{thm}\label{pb} Let $v\in D({\cal E})_{V_n,b}$ and $\{u_k\}_{k=0}^{\infty}\subset D({\cal E})_{V_n}$
satisfying $u_k$ converges to $u_0$ w.r.t. the $\tilde{\cal E}^{1/2}_1$-norm as $k\rightarrow\infty$.
Then there exists a subsequence $\{k'\}$ such that for ${\cal E}{\textrm{-}q.e.}$ $x\in V_n$,
\begin{eqnarray*}
& &P^{V_n}_x(\lim_{k'\rightarrow\infty}\int_0^{t}\tilde
v(X^{V_n}_s)dN^{n,[u_{k'}]}_s
=\int_0^{t}\tilde v(X^{V_n}_s)dN^{n,[u_0]}_s\\
& & \ \ \ \ \ \ \ \ \ \ \ \ \ \ {\rm uniformly\ on\ any\ finite\
interval\ of}\ t)=1.
\end{eqnarray*}
\end{thm}
\begin{proof} By Definition \ref{00s}, we have
\begin{eqnarray*}
\int_0^{t}\tilde
v(X^{V_n}_s)dN^{n,[u_k]}_s&=&N^{n,[\lambda(u_k,v)^{\triangle}]}_{t}
-\int_0^{t}\lambda(u_k,v)^{\triangle}(X^{V_n}_s)ds\\
& &-\frac{1}{2}G(u_k,v)_{t}-\int_0^{t}(u_k^*-u_k)v(X^{V_n}_s)ds.
\end{eqnarray*}
For each term of the right hand side of the above equation, we can
prove that  there exists a subsequence which converges uniformly
on\ any\ finite\ interval\ of $t$. Below we will only give the
proof for the convergence of the third term. The convergence of
the other three terms can be proved similar to \cite[Theorem
2.7]{CMS12} by virtue of \cite[Lemmas 2.5 and A.6]{MMS}.

Recall that for $u\in D({\cal E})_{V_n}$ and $v\in D({\cal
E})_{{V_n},b}$, $G(u,v)$ denotes the unique element in
$A_c^{n,+}-A_c^{n,+}$ that is corresponding to
$\tilde{\mu}^{(n)}_{\langle v,u^*\rangle}$ under the Revuz
correspondence between smooth measures of $({\cal E}^{V_n},D({\cal
E})_{V_n})$ and PCAFs of $X^{V_n}$. We use  $G^+(u,v)$ and
$G^-(u,v)$ to denote the PCAFs corresponding to
$\tilde{\mu}^{(n),+}_{\langle v,u^*\rangle}$ and
$\tilde{\mu}^{(n),-}_{\langle v,u^*\rangle}$, respectively.

 Define
$$
\hat{S}^{n,*}_{00}:=\{\mu\in S_{0}^n\,|\,\hat{U}^{V_n}_{1}\mu\leq
c\hat{G}_1^{V_n}\phi\ \mbox{for some constant}\ c>0\},
$$
where $S_0^n$ and $\hat{U}^{V_n}_{1}\mu$ denote respectively the
family of positive measures of finite energy integral and
1-co-potential relative to
$(\mathcal{E}^{V_n},D(\mathcal{E})_{V_n})$. By \cite[Theorem
A.3]{MMS}, for $A\in {\cal B}(E)$, if $\mu(A)=0$ for all $\mu\in
\hat{S}^{n,*}_{00}$ then $cap_{\phi}(A) = 0$.

 Let $\nu\in \hat{S}^{n,*}_{00}$. Then, there exists a positive constant
$C_{\nu}>0$ such that (cf. \cite[Lemma A.9]{MMS})
\begin{eqnarray*}
& &E^{V_n}_{\nu}\left[\sup_{0\le s\le t}|G(u_k,v)_s-G(u,v)_s|\right]\\
&&\ \ \ \ \ \ \ \ =E^{V_n}_{\nu}\left[\sup_{0\le s\le t}|G(u_k-u,v)_s|\right]\\
&&\ \ \ \ \ \ \ \ \le E^{V_n}_{\nu}[|G^+(u_k-u,v)_{t}|]+E^{V_n}_{\nu}[|G^-(u_k-u,v)_{t}|]\\
&&\ \ \ \ \ \ \ \ \le(1+t)C_{\nu}\int_{V_n}\widetilde{\bar{h}_n}d|\tilde{\mu}^{(n)}_{\langle v,(u_k-u)^*\rangle}|\\
&&\ \ \ \ \ \ \ \
\le(1+t)C_{\nu}\left(\int_{V_n}\widetilde{\bar{h}_n}^2d\tilde{\mu}^{(n)}_{\langle
v\rangle}\right)^{\frac{1}{2}}
\left(\int_{V_n}d\tilde{\mu}^{(n)}_{\langle (u_k-u)^*\rangle}\right)^{\frac{1}{2}}\\
&&\ \ \ \ \ \ \ \
\le2(1+t)C_{\nu}\|\widetilde{\bar{h}_n}\|_{\infty}(\eta^{(n)}(v,v))^{\frac{1}{2}}(\eta^{(n)}((u_k-u)^*,(u_k-u)^*))^{\frac{1}{2}},
\end{eqnarray*}
which converges to 0 as $k\rightarrow\infty$. The proof is
completed by the same method used in the proof of \cite[Lemma
5.1.2]{Fu94} (cf. \cite[Theorem 2.3.8]{oshima}).
\end{proof}

Similar to \cite[Proposition 2.6 and Corollary 3.2]{CMS12}, we can
prove the following propositions.
\begin{prop}\label{pa} Let $u,v\in D({\cal E})_{V_n,b}$. Then
$$
\int_0^t\tilde v(X^{V_n}_s)dN^{n,[u]}_s+\int_0^t\tilde
u(X^{V_n}_s)dN^{n,[v]}_s=N^{n,[uv]}_t-\langle
M^{n,[u]},M^{n,[v]}\rangle_t,\ \ t\ge 0,
$$
$P^{V_n}_{x}{\textrm{-}a.s.}$ for ${\cal E}{\textrm{-}q.e.}\ x\in
V_n$.
\end{prop}
\begin{prop}\label{cor1} Let $u\in D({\cal E})_{{V_n},b}$ and $\{v_k\}_{k=0}^{\infty}\subset D({\cal E})_{{V_n},b}$ satisfying $v_k$ converges to $v_0$ w.r.t. the $\|\cdot\|_{\infty}$-norm and the $\tilde{\cal E}^{1/2}_1$-norm as $k\rightarrow\infty$. Then there exists a subsequence $\{k'\}$ such that for ${\cal E}{\textrm{-}q.e.}$ $x\in V_n$,
\begin{eqnarray*}
&
&P^{V_n}_x(\lim_{k'\rightarrow\infty}\int_0^{t}\widetilde{v_{k'}}(X^{V_n}_s)dN^{n,[u]}_s
=\int_0^{t}{\widetilde v_0}(X^{V_n}_s)dN^{n,[u]}_s\\
& & \ \ \ \ \ \ \ \ \ \ \ \ \ \ {\rm uniformly\ on\ any\ finite\
interval\ of}\ t)=1.
\end{eqnarray*}
\end{prop}

\begin{defin} Let $u\in D({\cal E})_{V_n}$ and $v\in D({\cal E})_{V_n,b}$. We define for $t\ge 0$,
$$
\int_0^t\tilde v(X^{V_n}_{s-})dA^{n,[u]}_s:=\int_0^t\tilde
v(X^{V_n}_{s-})dM^{n,[u]}_s+\int_0^t\tilde
v(X^{V_n}_{s-})dN^{n,[u]}_s.
$$
\end{defin}
Finally, by virtue of \cite[Theorem 3.1]{MSW14} and similar to
\cite[Theorem 3.4]{CMS12}, we can prove the following result.
\begin{thm}\label{thm2} (i)  Let $u,v\in D({\cal E})_{V_n,b}$. Then,
\begin{eqnarray}\label{sd1}
& &\tilde u\tilde v(X^{V_n}_t)-\tilde u\tilde v(X^{V_n}_0)=\int_0^t\tilde v(X^{V_n}_{s-})
d\tilde u(X^{V_n}_s)+\int_0^t\tilde u(X^{V_n}_{s-})d\tilde v(X^{V_n}_s)\nonumber\\
& &\ \ \ \ \ \ +\langle M^{n,[u],c},M^{n,[v],c}\rangle_t\nonumber\\
& &\ \ \ \ \ \ +\sum_{0<s\le t}[\Delta (uv)(X^{V_n}_s)-\tilde
v(X^{V_n}_{s-})\Delta u(X^{V_n}_s)-\tilde u(X^{V_n}_{s-})\Delta
v(X^{V_n}_s)]
\end{eqnarray}
on $[0,\infty)$ $P^{V_n}_{x}{\textrm{-}a.s.}$ for ${\cal
E}{\textrm{-}q.e.}\ x\in V_n$.

(ii) Let $\Phi\in C^2(\mathbb{R}^n)$ and $u_1,\dots,u_n\in D({\cal
E})_{V_n,b}$. Then,
\begin{eqnarray*}
& &\Phi(\tilde u)(X^{V_n}_t)-\Phi(\tilde
u)(X^{V_n}_0)=\sum_{i=1}^n\int_0^t\Phi_i(\tilde
u(X^{V_n}_{s-}))dA^{n,[u_i]}_s\nonumber\\
& &\ \ \ \ \ \ \ \ \ \
+\frac{1}{2}\sum_{i,j=1}^n\int_0^t\Phi_{ij}(\tilde u(X^{V_n}_s))
d\langle M^{n,[u_i],c},M^{n,[u_j],c}\rangle_s\nonumber\\
& &\ \ \ \ \ \ \ \ \ \ +\sum_{0<s\le t}\left[\Delta \Phi(\tilde
u(X^{V_n}_s))-\sum_{i=1}^n\Phi_i(\tilde u(X^{V_n}_{s-}))\Delta
u_i(X^{V_n}_s)\right]
\end{eqnarray*}
on $[0,\infty)$ $P^{V_n}_{x}{\textrm{-}a.s.}$ for ${\cal
E}{\textrm{-}q.e.}\ x\in V_n$, where
$$
\Phi_i(x)=\frac{\partial\Phi}{\partial x_i}(x),\ \ \Phi_{ij}(x)=\frac{\partial^2\Phi}{\partial x_i\partial x_j}(x),\ \ i,j=1,\dots,n,
$$
and $u=(u_1,\dots,u_n)$.
\end{thm}

\subsection {Stochastic Integral for $X$}
 In this subsection, for $u,v\in D(\mathcal{E})_{loc}$, we will define the
stochastic integral $\int_0^t\tilde v(X_{s-})dA_s^{[u]}$. To this
end, we first choose a $\{{V_n}\}\in \Theta$ such that Assumption
\ref{assum1q} is satisfied and $\widetilde{\hat{h}}$ is bounded on
each ${V_n}$. Then, we choose $\{E_{n}\}\in\Theta$ and
$\{u_n,v_n\}$ such that $E_n\subset V_n$, $u_n,v_n\in
D(\mathcal{E})_{V_n,b}$, $u=u_n$ and $v=v_n$ on $E_n$ for each
$n\in \mathbb{N}$. The existence of $\{E_{n}\}$ and $\{u_n,v_n\}$
is justified by the argument before Lemma \ref{lem2.3}. Now we
define $\int_0^t\tilde v(X_{s-})dA_s^{[u]}$ on $I(\zeta)$ by
\begin{equation}\label{equ11}
\int_0^t\tilde v(X_{s-})dA_s^{[u]}:=\lim_{n\rightarrow\infty}
\int_0^t{\widetilde v_n}(X_{s-})dA_s^{n,[u_n]},\end{equation}
where the stochastic integral $\int_0^t{\widetilde
v_n}(X_{s-})dA_s^{n,[u_n]}$ is defined as in the above subsection.

\begin{thm} For $u,v\in D(\mathcal{E})_{loc}$, the stochastic integral in (\ref{equ11}) is
well-defined. Moreover, if $u,u',v,v'\in D(\mathcal{E})_{loc}$
satisfying $u=u'$ and $v=v'$ on $U$ for some finely open set $U$,
then
\begin{equation}\label{exten}
\int_0^t\tilde v(X_{s-})dA_s^{[u]}=\int_0^t{\widetilde
v'}(X_{s-})dA_s^{[u']},
\end{equation}
for $0\le t<\tau_{U}$ if $\tau_{U}<\zeta$; and for $0\le
t\le\zeta$ if $\tau_{U}=\zeta$, $P_{x}{\textrm{-}a.s.}$ for ${\cal
E}{\textrm{-}q.e.}\ x\in E$.
\end{thm}

\begin{proof} First, we fix a $\{{V_n}\}\in \Theta$ such that Assumption
\ref{assum1q} is satisfied and $\widetilde{\hat{h}}$ is bounded on
each ${V_n}$. Suppose that there are two finely open sets $F_k$,
$F_l$ satisfying $F_k\subset V_k$, $F_l\subset V_l$, $k<l$;
$f_k,g_k\in D(\mathcal{E})_{V_k,b}$, $u=f_k$, $v=g_k$ on $F_k$;
$f_l,g_l\in D(\mathcal{E})_{V_l,b}$, $u=f_l$, $v=g_l$ on $F_l$.
Below we will show that
\begin{equation}\label{gh1}
\int_0^t{\widetilde g_k}(X_{s-})dA_s^{k,[f_k]}=\int_0^t{\widetilde
g_l}(X_{s-})dA_s^{l,[f_l]},
\end{equation}
for $0\le t<\tau_{F_k\cap F_l}$ if $\tau_{F_k\cap F_l}<\zeta$; and
for $0\le t\le\zeta$ if $\tau_{F_k\cap F_l}=\zeta$,
$P_{x}{\textrm{-}{\rm a.s.}}$ for ${\cal E}{\textrm{-}{\rm q.e.}}\
x\in V_k$.

In fact, by approximating $f_l$ by a sequence of functions in
$D(L^{V_l})$, Proposition \ref{pro2.7} and Theorem \ref{pb}, we
find that
\begin{equation}\label{gh2}
\int_0^t{\widetilde g_k}(X_{s-})dA_s^{l,[f_l]}=\int_0^t{\widetilde
g_l}(X_{s-})dA_s^{l,[f_l]},\ \ 0\le t\le \tau_{F_k\cap F_l},
\end{equation}
$P^{V_l}_{x}{\textrm{-}{\rm a.s.}}$ for ${\cal E}{\textrm{-}{\rm
q.e.}}\ x\in V_l$. Since $A_{t\wedge\tau_{V_l}}^{l,[f_l]}\in{\cal
F}^{V_l}_{t\wedge\tau_{V_l}-}$, (\ref{gh1}) holds
$P_{x}{\textrm{-}{\rm a.s.}}$ for ${\cal E}{\textrm{-}{\rm q.e.}}\
x\in V_l$.

Further, we obtain by (\ref{sd1}) that
\begin{equation}\label{gh3}
\int_0^t{\widetilde g_k}(X_{s-})dA_s^{l,[f_k]}=\int_0^t{\widetilde
g_k}(X_{s-})dA_s^{l,[f_l]},
\end{equation}
for $0\le t<\tau_{F_k\cap F_l}$ if $\tau_{F_k\cap F_l}<\zeta$; and
for $0\le t\le\zeta$ if $\tau_{F_k\cap F_l}=\zeta$,
$P^{V_l}_{x}{\textrm{-}{\rm a.s.}}$ and hence
$P_{x}{\textrm{-}{\rm a.s.}}$ for ${\cal E}{\textrm{-}{\rm q.e.}}\
x\in V_l$. Note that
$M^{k,[f_k]}_{t\wedge\tau_{F_k}}=M^{l,[f_k]}_{t\wedge\tau_{F_k}}$
and
$N^{k,[f_k]}_{t\wedge\tau_{F_k}}=N^{l,[f_k]}_{t\wedge\tau_{F_k}}$
$P^{V_k}_{x}{\textrm{-}{\rm a.s.}}$ for ${\cal E}{\textrm{-}{\rm
q.e.}}\ x\in V_k$ (cf. the proof of \cite[Lemma 1.14]{MSW14}). By
approximating $f_k$ by a sequence of functions in $D(L^{V_k})$,
Proposition \ref{pro2.7} and Theorem \ref{pb}, we get
\begin{equation}\label{gh4}
\int_0^t{\widetilde g_k}(X_{s-})dA_s^{k,[f_k]}=\int_0^t{\widetilde
g_k}(X_{s-})dA_s^{l,[f_k]},\ \ 0\le t\le \tau_{F_k},
\end{equation}
$P^{V_k}_{x}{\textrm{-}{\rm a.s.}}$ and hence
$P_{x}{\textrm{-}{\rm a.s.}}$ for ${\cal E}{\textrm{-}{\rm q.e.}}\
x\in V_k$. Therefore, (\ref{gh1}) holds for $0\le t<\tau_{F_k\cap
F_l}$ if $\tau_{F_k\cap F_l}<\zeta$; and for $0\le t\le\zeta$ if
$\tau_{F_k\cap F_l}=\zeta$, $P_{x}{\textrm{-}{\rm a.s.}}$ for
${\cal E}{\textrm{-}{\rm q.e.}}\ x\in V_k$  by
(\ref{gh2})-(\ref{gh4}).

Now we suppose that (\ref{equ11}) is defined by a different
$\{V_n\}\in \Theta$, say $\{V'_n\}\in \Theta$. By considering
$\{V_n\cap V'_n\}$,  \cite[Proposition 2.4]{MSW14} and the above
argument, we find that the two limits in (\ref{equ11}) are equal
on $I(\zeta)$, $P_{x}{\textrm{-}{\rm a.s.}}$ for ${\cal
E}{\textrm{-}{\rm q.e.}}\ x\in E$. Therefore, (\ref{equ11}) is
well-defined.

From (\ref{gh1}) and its proof, we find that if $u,u',v,v'\in
D(\mathcal{E})_{loc}$ satisfying $u=u'$ and $v=v'$ on $U$ for some
finely open set $U$, then there exists an $\{E_n\}\in\Theta$ such
that (\ref{exten}) holds on $\cup_n [\![0,\tau_{E_n\cap U}]\!]$,
$P_{x}{\textrm{-}{\rm a.s.}}$ for ${\cal E}{\textrm{-}{\rm q.e.}}\
x\in E$. By \cite[Proposition 2.4]{MSW14}, this implies that
(\ref{exten}) holds for $0\le t<\tau_{U}$ if $\tau_{U}<\zeta$; and
for $0\le t\le\zeta$ if $\tau_{U}=\zeta$, $P_{x}{\textrm{-}{\rm
a.s.}}$ for ${\cal E}{\textrm{-}{\rm q.e.}}\ x\in E$. The proof is
complete.
\end{proof}

From the proof of Theorem \ref{thm3.2}, we find that $M^{[u],c}$
is well defined whenever $u\in D({\cal E})_{loc}$. Therefore, we
obtain by Theorem \ref{thm2} and (\ref{exten}) the following
theorem.
\begin{thm} Let $\Phi\in C^2(\mathbb{R}^n)$ and $u_1,\dots,u_n\in D({\cal
E})_{loc}$. Then,
\begin{eqnarray}\label{Ito1}
& &A_t^{[\Phi(u)]}=\sum_{i=1}^n\int_0^t\Phi_i(\tilde
u(X_{s-}))dA^{[u_i]}_s+\frac{1}{2}\sum_{i,j=1}^n\int_0^t\Phi_{ij}(\tilde
u(X_s))
d\langle M^{[u_i],c},M^{[u_j],c}\rangle_s\nonumber\\
& &\ \ \ \ +\sum_{0<s\le t}\left[\Delta \Phi(\tilde
u(X_s))-\sum_{i=1}^n\Phi_i(\tilde u(X_{s-}))\Delta u_i(X_s)\right]
\end{eqnarray}
on $I(\zeta)$ $P_{x}{\textrm{-}a.s.}$ for ${\cal
E}{\textrm{-}q.e.}\ x\in E$, where
$$
\Phi_i(x)=\frac{\partial\Phi}{\partial x_i}(x),\ \
\Phi_{ij}(x)=\frac{\partial^2\Phi}{\partial x_i\partial x_j}(x),\
\ i,j=1,\dots,n,
$$
and $u=(u_1,\dots,u_n)$.
\end{thm}

\section {Some Examples\setcounter{equation}{0}}
In this section, we give some examples for which  all results of
the previous two sections can be applied.

First, we consider a local semi-Dirichlet form.

\begin{exa}\label{e2} (see \cite{smuland}) Let $d\geq3$, $U$ be an open subset of $\mathbb{R}^d$, $\sigma,\rho\in L^{1}_{loc}(U;dx)$, $\sigma,\rho>0$ $dx{\textrm{-}}a.e.$
For $u,v\in C^{\infty}_0(U)$, we define
\begin{eqnarray*}
\mathcal{E}_{\rho}(u,v)=\sum_{i,j=1}^d\int_U\frac{\partial
u}{\partial x_i}
 \frac{\partial v}{\partial x_j}\rho dx.
 \end{eqnarray*}
 Assume that
  $$(\mathcal{E}_{\rho},C^{\infty}_0(U)) \ \mbox{is closable on}\ L^2(U;\sigma dx).$$

Let $a_{ij},b_i,d_i\in L^1_{loc}(U;dx)$, $1\leq i,j\leq d$. For
$u,v\in C^{\infty}_{0}(U)$, we define
\begin{eqnarray*}
\mathcal{E}(u,v) &=&\sum_{i,j=1}^{d}\int_U\frac{\partial u}{\partial
     x_{i}}\frac{\partial u}{\partial
     x_{j}}a_{ij}dx+\sum_{i=1}^{d}\int_U
     \frac{\partial u}{\partial x_{i}}vb_{i}dx\\
     & &+\sum_{i=1}^{d}\int_U
     u\frac{\partial v}{\partial x_{i}}d_{i}dx
   +\int_U uvcdx.
\end{eqnarray*}
Set $\tilde{a}_{ij}:=\frac{1}{2}(a_{ij}+a_{ji})$,
$\check{a}_{ij}:=\frac{1}{2}(a_{ij}-a_{ji})$,
$\underline{b}:=(b_1,\dots, b_d)$, and $\underline{d}:=(d_1,\dots,
d_d)$. Define F to be the set of all functions $g\in
L^1_{loc}(U;dx)$ such that the distributional derivatives
$\frac{\partial g}{\partial x_i},\ 1\leq i\leq d$, are in
$L^1_{loc}(U;dx)$ such that $\|\nabla g\|(g\sigma)^{-\frac{1}{2}}\in
L^{\infty}(U;dx)$ or $\|\nabla
g\|^p(g^{p+1}\sigma^{p/q})^{-\frac{1}{2}}\in L^d(U;dx)$ for some
$p,q\in (1,\infty)$ with $\frac{1}{p}+\frac{1}{q}=1,\ p<\infty$,
where $\|\cdot\|$ denotes Euclidean distance in $\mathbb{R}^d$. We
say that a $\mathcal{B}(U)-$measurable function f has property
$(A_{\rho,\sigma})$ if one of the following conditions holds:

(i) $f(\rho\sigma)^{-\frac{1}{2}}\in L^{\infty}(U;dx)$.

(ii) $f^p(\rho^{p+1}\sigma^{p/q})^{-\frac{1}{2}}\in L^d(U,dx)$ for
some
 $p,q\in (1,\infty)$ with ${1\over p}+{1\over q}=1,\ p<\infty,$ and $\rho\in F$.

Suppose that

(A.I) There exists $\eta>0$ such that $ \sum_{i,j=1}^d
\tilde{a}_{ij}\xi_i\xi_j\ge\eta|\underline{\xi}|^2$, $\forall
\underline{\xi}=(\xi_1,\dots,\xi_d)\in\mathbb{R}^d$.

(A.II) $\check{a}_{ij}\rho^{-1}\in L^{\infty}(U;dx)$ for $1\le
i,j\le d$.

(A.III) For all $K\subset U$, $K$ compact,
$1_K\|\underline{b}+\underline{d}\|$
 and $1_Kc^{1/2}$ have property $(A_{\rho,\sigma}),$ and
 $(c+\alpha_0\sigma) dx-\sum_{i=1}^d\frac{\partial d_i}{\partial x_i}$
 is a positive measure on $\mathcal{B}(U)$ for some $\alpha_0\in (0,\infty)$.

(A.IV) $||\underline{b}-\underline{d}||$ has property
$(A_{\rho,\sigma})$.

(A.V) $\underline{b}=\underline{\beta}+\underline{\gamma}$ such that
$\|\underline\beta\|,
 \|\underline\gamma\|\in L^{1}_{loc}(U,dx)$,
 $(\alpha_0\sigma+c)dx-\sum_{1}^{d}{\partial \gamma_{i}\over \partial x_{i}}$
 is a positive measure on $\mathcal{B}(U)$ and
 $\|\underline{\beta}\|$
 has property $(A_{\rho,\sigma})$.

\noindent Then, by \cite[Theorem 1.2]{smuland}, there exists
$\alpha>0$ such that $(\mathcal{E}_{\alpha}, C^{\infty}_{0}(U))$ is
closable on $L^{2}(U;dx)$ and its closure
$(\mathcal{E}_{\alpha},D(\mathcal{E}_{\alpha}))$ is a regular local
semi-Dirichlet form on $L^{2}(U;dx)$. Define
$\eta_{\alpha}(u,u):=\mathcal{E}_{\alpha}(u,u)-\int\langle\triangledown
u,\underline{\beta}\rangle udx$
 for $u\in D(\mathcal{E}_{\alpha})$. By \cite[Theorem 1.2 (ii) and (1.28)]{smuland}, we know
$(\eta_{\alpha},D(\mathcal{E})_{\alpha})$ is a Dirichlet form and
there exists $C>1$ such that for any $u\in D(\mathcal{E}_{\alpha})$,
\begin{eqnarray*}
\frac{1}{C}\eta_{\alpha}(u,u)\leq \mathcal{E}_{\alpha}(u,u)\leq
C\eta_{\alpha}(u,u).
\end{eqnarray*}
 Let $X$ be the diffusion process associated with $(\mathcal{E}_{\alpha},D(\mathcal{E}_{\alpha}))$. For $u\in D(\mathcal{E}_{\alpha})_{loc}$,
 we have the decomposition (\ref{Vs2}) and It\^{o}'s formula (\ref{Ito1}).
 \end{exa}

Next we consider a  semi-Dirichlet form of pure jump type.

 \begin{exa}(See \cite{FU12} and cf. also \cite{SW}) Let $(E,d)$ be a locally compact separable metric space, $m$ be a positive Radon Measure on $E$
  with full topological support, and $k(x,y)$ be a nonnegative Borel measurable function on $\{(x,y)\in E\times E|\ x\neq y\}$.
  Set
$k_s(x,y)=\frac{1}{2}(k(x,y)+k(y,x))$ and $k_a=\frac{1}{2}(k(x,y)-k(y,x))$.
 Denote by $C^{lip}_0(E)$ the family of all uniformly Lipschitz continuous functions on $E$ with compact support. Suppose that the
 following conditions hold:\\
 (B.I)\ $x\mapsto\int_{y\neq x}(1\wedge d(x,y)^2)k_s(x,y)m(dy)\in L^1_{loc}(E;dx).$\\
(B.II)\ $sup_{x\in E}\int_{\{y:\
k_s(x,y)\neq0\}}\frac{k^2_a(x,y)}{k_s(x,y)}m(dy)<\infty.$

 Define
for $u,v\in C^{lip}_0(E)$,
\begin{eqnarray*}
\eta(u,v)&=&\int\hskip-0.2cm\int_{x\neq
y}(u(x)-u(y))(v(x)-v(y))k_s(x,y)m(dx)m(dy),
\end{eqnarray*}
and
\begin{eqnarray*}
\mathcal{E}(u,v)&=&\frac{1}{2}\eta(u,v)+\int\hskip-0.2cm\int_{x\neq
y}(u(x)-u(y))(v(x)-v(y))k_a(x,y)m(dx)m(dy).
\end{eqnarray*}
Then, there exists $\alpha>0$ such that
$(\mathcal{E}_{\alpha},C^{lip}_0(E))$ is closable on $L^2(E;dx)$
and its closure $(\mathcal{E}_{\alpha},D(\mathcal{E}_{\alpha}))$
is a regular semi-Dirichlet form on $L^2(E,dx)$. Moreover, there
exists $C>1$ such that for any $u\in D(\mathcal{E}_{\alpha})$,
\begin{eqnarray*}
\frac{1}{C}\eta_{\alpha}(u,u)\leq\mathcal{E}_{\alpha}(u,u)\leq
C\eta_{\alpha}(u,u).
\end{eqnarray*}
 Let $X$ be the pure jump process associated with $(\mathcal{E}_{\alpha},D(\mathcal{E}_{\alpha}))$.
 For $u\in D(\mathcal{E}_{\alpha})_{loc}$, we have the decomposition (\ref{new3}) and It\^{o}'s formula (\ref{Ito1}).
\end{exa}

Finally, we consider a general semi-Dirichlet form with diffusion,
jumping and killing terms.
\begin{exa}(See \cite{uemura})
Let $G$ be an open set of $\mathbb{R}^d$. Suppose that the
following conditions hold:

\noindent (C.I) There exist $0<\lambda\le \Lambda$ such that
$$
\lambda|\xi|^2\le \sum_{i,j=1}^da_{ij}(x)\xi_i\xi_j\le
\Lambda|\xi|^2\ \ {\rm for}\ x\in G,\ \xi\in\mathbb{R}^d.
$$

\noindent (C.II) $b_i\in L^{d}(G;dx)$, $i=1,\dots,d$.

\noindent (C.III) $c\in L^{d/2}_+(G;dx)$.

\noindent (C.IV)
$x\mapsto\int_{y\not=x}(1\wedge|x-y|^2)k_s(x,y)dy\in
L^1_{loc}(G;dx)$.

\noindent (C.V) $ \sup_{x\in G}\int_{\{|x-y|\ge 1,y\in
G\}}|k_a(x,y)|dy<\infty$, $ \sup_{x\in G}\int_{\{|x-y|< 1,y\in
G\}}|k_a(x,y)|^{\gamma}dy<\infty$ for some $0<\gamma\le 1$, and
$|k_a(x,y)|^{2-\gamma}\le C_1k_s(x,y)$, $x,y\in G$ with
$0<|x-y|<1$ for some constant $C_1>0$.

Define for $u,v\in C^1_{0}(G)$,
\begin{eqnarray*}
\eta(u,v)&=&\frac{1}{2}\sum_{i=1}^d\int_G\frac{\partial
u}{\partial x_i}(x)\frac{\partial v}{\partial
x_i}(x)dx\\
& &\ \ \ \
+\frac{1}{2}\int\hskip-0.2cm\int_{x\not=y}(u(x)-u(y))(v(x)-v(y))k_s(x,y)dxdy
\end{eqnarray*}
and
\begin{eqnarray*}
{\cal
E}(u,v)&=&\frac{1}{2}\sum_{i=1}^d\int_Ga_{ij}(x)\frac{\partial
u}{\partial x_i}(x)\frac{\partial v}{\partial
x_j}(x)dx+\sum_{i=1}^d\int_Gb_{i}(x)u(x)\frac{\partial v}{\partial
x_i}(x)dx\\
& &\ \ \ \ +\int_Gu(x)v(x)c(x)dx\\
& &\ \ \ \ +
\frac{1}{2}\int\hskip-0.2cm\int_{x\not=y}(u(x)-u(y))(v(x)-v(y))k_s(x,y)dxdy\\
& &\ \ \ \
+\int\hskip-0.2cm\int_{x\not=y}(u(x)-u(y))v(x)k_a(x,y)dxdy.
\end{eqnarray*}
Then, when $\lambda$ is sufficiently large, there exists
$\alpha>0$ such that $({\cal E}_{\alpha},C^1_{0}(G))$ is closable
on $L^2(G;dx)$ and its closure $({\cal E}_{\alpha},D({\cal
E}_{\alpha}))$ is a regular semi-Dirichlet form on $L^2(G;dx)$.
Moreover, there exists $C'>1$ such that for any $u\in
D(\mathcal{E}_{\alpha})$,
\begin{eqnarray*}
\frac{1}{C'}\eta_{\alpha}(u,u)\leq {\cal E}_{\alpha}(u,u)\leq
C'\eta_{\alpha}(u,u).
\end{eqnarray*}
 Let $X$ be the Markov process associated with $(\mathcal{E}_{\alpha},D(\mathcal{E}_{\alpha}))$. For $u\in D(\mathcal{E}_{\alpha})_{loc}$,
 we have the decomposition (\ref{new3}) and It\^{o}'s formula (\ref{Ito1}).
\end{exa}
\bigskip

{ \noindent {\bf\large Acknowledgments} \vskip 0.1cm  \noindent We
acknowledge the support of NSFC (Grant No. 11361021 and 11201102),
Natural Science Foundation of Hainan Province (Grant No. 113007)
and NSERC (Grant No. 311945-2013).} We thank Professor Zhi-Ming Ma
and Dr. Li-Fei Wang for discussions that improved the presentation
of this paper.


\begin{thebibliography}{a}

\bibitem{f} Albeverio, S.,  Fan, R.Z.,  R\"ockner, M., Stannat, W.: A remark on coercive forms and associated semigroups,
{Partial Differential Operators and Mathematical Physics, Operator
Theory Advances and Applications} {\bf 78}, 1-8 (1995)

\bibitem{CMS12} Chen, C.Z.,  Ma, L.,  Sun, W.: Stochastic calculus for
Markov processes associated with non-symmetric Dirichlet forms.
{Science China Mathematics}  \textbf{55}, 2195-2203 (2012)

\bibitem{CFKZ} Chen, Z.Q., Fitzsimmons, P.J., Kuwae, K.,  Zhang, T.S.: Stochastic calculus for symmetric
Markov processes. {Ann. Probab.} \textbf{36}, 931-970 (2008)


\bibitem{F1} F\"ollmer, H.: Calcul d'It\^{o} sans probabilit\'es. Seminar on Probability, XV (Univ. Strasbourg, Strasbourg,
1979/1980) pp. 143-150, Lect. Notes in Math., 850, Springer,
Berlin (1981)

\bibitem{F2} F\"ollmer, H.: Dirichlet processes, Stochastic integrals (Proc. Sympos., Univ. Durham, Durham, 1980), pp.
476-478, Lect. Notes in Math., 851, Springer, Berlin (1981)

\bibitem{Fu79} { Fukushima, M.}: A decomposition of additive functionals of finite energy. { Nagoya Math. J.} {\bf 74}, 137-168 (1979)

\bibitem{Fu94} { Fukushima, M., Oshima, Y., Takeda,  M.}: {{Dirichlet Forms
               and Symmetric Markov Processes}}. second revised and extended edition, Walter de Gruyter
               (2011)

\bibitem{FU12} Fukushima, M., Uemura, T.: Jump-type Hunt processes generated by lower bounded semi-Dirichlet forms. { Ann. Probab.}
{\bf 40}, 858-889 (2012)

\bibitem{HWY} He, S.W., Wang, J.G., Yan, J.A.: {Semimartingale Theory and Stochastic
Calculus}. Science Press, Beijing (1992)

\bibitem{HC06} Hu, Z.C., Ma, Z.M., Sun, W.:  Extensions of L\'{e}vy-Khintchine formula
and Beurling-Deny formula in semi-Dirichlet forms setting.  { J.
Funct. Anal.} {\bf 239}, 179-213 (2006)



\bibitem{kuwae} Kuwae, K.: Maximum principles for subharmonic functions via local semi-Dirichlet
forms. {Can. J. Math.} {\bf 60}, 822-874 (2008)

\bibitem{kuwae2010} Kuwae, K.: Stochastic calculus over symmetric Markov processes without time reversal. { Ann. Probab.} {\bf 38}, 1532-1569 (2010)


\bibitem{MMS} Ma, L., Ma, Z.M., Sun, W.: Fukushima's decomposition for diffusions
associated with semi-Dirichlet forms. {Stoch. Dyn.} {\bf 12},
1250003 (2012)

\bibitem{MOR} Ma, Z.M., Overbeck L., R\"ockner, M.:  Markov processes associated with semi-Dirichlet forms. { Osaka J. Math.} {\bf 32}, 97-119 (1995)

\bibitem{MR92} Ma, Z.M., R\"ockner, M.: { Introduction to the Theory
              of (Non-Symmetric) Dirichlet Forms}. Springer-Verlag
              (1992)

\bibitem{MR95} { Ma, Z.M.,  R\"{o}ckner,M.}: Markov processes associated with positivity preserving
             coercive forms, { Can. J. Math.} {\bf 47}, 817-840 (1995)

\bibitem{MSW14} Ma, Z.M., Sun W., Wang, L.F.: Fukushima type decomposition for semi-Dirichlet forms. Preprint,
http://www.mathstat.concordia.ca/resources-and-links/technical-reports/index.php
(2014)

\bibitem{N} Nakao, S.: Stochastic calculus for continuous additive functionals of zero energy. { Z. Wahrsch. verw. Gebiete} {\bf 68}, 557-578 (1985)

\bibitem{oshima} Oshima, Y.: {Lecture on Dirichlet Spaces}. Univ. Erlangen-N\"urnberg
(1988)

\bibitem{oshima13} { Oshima, Y.}: {{Semi-Dirichlet Forms and Markov Processes}}.  Walter de Gruyter
(2013)


\bibitem{smuland} R\"ockner, M., Schmuland, B.: Quasi-regular Dirichlet forms:
             examples and counterexamples. { Can. J. Math.} {\bf 47}, 165-200 (1995)

\bibitem{SW} {  Schilling, R.L., Wang, J.}: Lower bounded semi-Dirichlet forms associated with L\'evy type
operators. {arXiv:1108.3499} (2012)


\bibitem{uemura} Uemura, T.: On multidimensional diffusion processes with jumps. {Osaka J. Math.}, to appear

\bibitem{W} { Wang, L.F.}: Fukushima's decomposition of semi-Dirichlet forms and some related
research. Ph.D. thesis, Chinese Academy of Sciences (2013)

\bibitem{Alexander 2011} Walsh, A.:
Extended It\^o calculus. Ph.D. thesis, Univ. Pierre et Marie Curie
(2011)

\bibitem{Alexander 2012} Walsh, A.: Extended It\^o calculus for symmetric Markov processes.
{ Bernoulli} {\bf 18}, 1150-1171 (2012)

\bibitem{Alexander 2013} Walsh, A.:
Stochastic integration with respect to additive functionals of
zero quadratic variation. { Bernoulli} {\bf 19}, 2414-2436 (2013)


\end{thebibliography}
\end{document}